\documentclass[10pt,letterpaper,twoside]{article}

\usepackage[letterpaper,left=1in,right=1in,top=1.5in,bottom=1.5in]{geometry}

\title{Spherical Witt vectors and integral models for spaces}
\author{Benjamin Antieau}
\date{\today}

\usepackage[pdfstartview=FitH,
            pdfauthor={},
            pdftitle={Transmutation},
            colorlinks,
            linkcolor=reference,
            citecolor=citation,
            urlcolor=e-mail,
            backref]{hyperref}

\usepackage{amsmath}
\usepackage{amscd}
\usepackage{amsbsy}
\usepackage{amssymb}
\usepackage{verbatim}
\usepackage{eufrak}
\usepackage{eucal}
\usepackage{microtype}
\usepackage{hyperref}
\usepackage{mathrsfs}
\usepackage{amsthm}
\usepackage{stmaryrd}
\usepackage{xy}
\input xy
\xyoption{all}
\usepackage{tikz}
\usetikzlibrary{matrix,arrows}
\usepackage{tikz-cd}
\usepackage{dsfont}
\usepackage[T1]{fontenc}
\usepackage{enumitem}
\setlist{noitemsep}

\usepackage{fancyhdr}
\usepackage[bbgreekl]{mathbbol}
\usepackage{amsfonts}
\DeclareSymbolFontAlphabet{\mathbb}{AMSb} 
\DeclareSymbolFontAlphabet{\mathbbl}{bbold}

\usepackage{yhmath}

\usepackage{color}
\definecolor{todo}{rgb}{1,0,0}
\definecolor{conditional}{rgb}{0,1,0}
\definecolor{e-mail}{rgb}{0,.40,.80}
\definecolor{reference}{rgb}{.20,.60,.22}
\definecolor{mrnumber}{rgb}{.80,.40,0}
\definecolor{citation}{rgb}{0,.40,.80}

\let\oldmarginpar\marginpar
\renewcommand\marginpar[1]{\-\oldmarginpar[\raggedleft\footnotesize #1]%
{\raggedright\footnotesize #1}}

\newcommand{\Cscr}{\mathcal{C}}
\newcommand{\Dscr}{\mathcal{D}}

\newcommand{\Fscr}{\mathcal{F}}

\newcommand{\Oscr}{\mathcal{O}}
\newcommand{\Pscr}{\mathcal{P}}

\newcommand{\Sscr}{\mathcal{S}}

\newcommand{\B}{\mathrm{B}}
\renewcommand{\c}{\mathrm{c}}

\newcommand{\D}{\mathrm{D}}

\newcommand{\F}{\mathrm{F}}

\renewcommand{\H}{\mathrm{H}}
\newcommand{\h}{\mathrm{h}}

\renewcommand{\L}{\mathrm{L}}

\newcommand{\R}{\mathrm{R}}

\newcommand{\T}{\mathrm{T}}
\renewcommand{\t}{\mathrm{t}}

\newcommand{\V}{\mathrm{V}}

\newcommand{\bA}{\mathbf{A}}

\newcommand{\bE}{\mathbf{E}}
\newcommand{\bF}{\mathbf{F}}
\newcommand{\bG}{\mathbf{G}}
\newcommand{\bH}{\mathbf{H}}

\newcommand{\bN}{\mathbf{N}}

\newcommand{\bQ}{\mathbf{Q}}

\newcommand{\bS}{\mathbf{S}}

\newcommand{\bW}{\mathbf{W}}

\newcommand{\bZ}{\mathbf{Z}}

\renewcommand{\mathds}{\mathbbl}

\newcommand{\WW}{\mathds{W}}

\newcommand{\mfrak}{\mathfrak{m}}

\newcommand{\Cat}{\Cscr\mathrm{at}}

\newcommand{\Fin}{\Fscr\mathrm{in}}

\newcommand{\Sp}{\Sscr\mathrm{p}}

\newcommand{\op}{\mathrm{op}}
\newcommand{\cofib}{\mathrm{cofib}}
\newcommand{\fib}{\mathrm{fib}}

\newcommand{\cn}{\mathrm{cn}}

\newcommand{\bWhat}{\widehat{\bW}}

\newcommand{\Mod}{\mathrm{Mod}}

\newcommand{\Perf}{\mathrm{Perf}}
\newcommand{\Ind}{\mathrm{Ind}}
\newcommand{\Pro}{\mathrm{Pro}}

\newcommand{\CAlg}{\mathrm{CAlg}}

\newcommand{\DAlg}{\mathrm{DAlg}}

\newcommand{\Gr}{\mathrm{Gr}}
\newcommand{\gr}{\mathrm{gr}}
\newcommand{\ins}{\mathrm{ins}}

\newcommand{\LSym}{\mathrm{LSym}}
\newcommand{\Sym}{\mathrm{Sym}}

\newcommand{\nlRL}{\mathrm{nlRL}}
\newcommand{\rcolon}{\,{:}\!\,}
\newcommand{\Rees}{\mathrm{Rees}}
\newcommand{\FD}{\mathrm{FD}}
\newcommand{\FDhat}{\widehat{\FD}}

\newcommand{\Stone}{\mathrm{Stone}}

\newcommand{\un}{\mathrm{un}}

\newcommand{\nil}{\mathrm{nil}}
\newcommand{\ft}{\mathrm{ft}}
\newcommand{\fin}{\mathrm{fin}}

\newcommand{\ccn}{\mathrm{ccn}}

\newcommand{\bFbar}{\overline{\bF}}
\newcommand{\Poly}{\mathrm{Poly}}
\newcommand{\Vect}{\mathrm{Vect}}
\newcommand{\perf}{\mathrm{perf}}

\newcommand{\syn}{\mathrm{syn}}
\newcommand{\ev}{\mathrm{ev}}

\newcommand{\heart}{\heartsuit}

\newcommand{\id}{\mathrm{id}}

\renewcommand{\geq}{\geqslant}
\renewcommand{\leq}{\leqslant}

\newcommand{\TC}{\mathrm{TC}}

\newcommand{\dR}{\mathrm{dR}}

\newcommand{\MU}{\mathrm{MU}}
 
\newcommand{\Map}{\mathrm{Map}}

\newcommand{\Hom}{\mathrm{Hom}}
\newcommand{\Fun}{\mathrm{Fun}}

\newcommand{\pic}{\mathrm{pic}}
\newcommand{\Pic}{\mathrm{Pic}}

\newcommand{\Gm}{\bG_{m}}

\DeclareMathOperator*{\colim}{colim}

\DeclareMathOperator*{\Tot}{Tot}

\newcommand{\et}{\mathrm{\acute{e}t}}

\newcommand{\PreStk}{\mathrm{PreStk}}
\newcommand{\Stk}{\mathrm{Stk}}
\DeclareMathOperator{\Spec}{Spec}

\newcommand{\red}{\mathrm{red}}

\newcommand{\we}{\simeq}
\newcommand{\iso}{\cong}

\theoremstyle{plain}
\newtheorem{theorem}{Theorem}[section]
\newtheorem*{theorem*}{Theorem}
\newtheorem{lemma}[theorem]{Lemma}

\newtheorem{proposition}[theorem]{Proposition}

\newtheorem{conjecture}[theorem]{Conjecture}
\newtheorem{corollary}[theorem]{Corollary}
\newtheorem*{corollary*}{Corollary}

\theoremstyle{plain}

\theoremstyle{definition}

\newtheoremstyle{named}{}{}{\itshape}{}{\bfseries}{.}{.5em}{#1 \thmnote{#3}}
\theoremstyle{named}

\theoremstyle{definition}
\newtheorem{definition}[theorem]{Definition}
\newtheorem{warning}[theorem]{Warning}
\newtheorem{variant}[theorem]{Variant}

\newtheorem{notation}[theorem]{Notation}

\newtheorem{example}[theorem]{Example}
\newtheorem*{example*}{Example}

\newtheorem*{question*}{Question}
\newtheorem{construction}[theorem]{Construction}

\newtheorem{remark}[theorem]{Remark}

\begin{document}

\maketitle

\begin{abstract}
    \noindent
    We give a new construction of the spherical Witt vector functor of Lurie and
    Burklund--Schlank--Yuan and extend it to nonconnective objects using synthetic spectra
    and recent work of Holeman. The spherical Witt vectors are used to build spherical versions of
    perfect $\lambda$-rings and to motivate new results in Grothendieck's
    schematization program, building on work of Ekedahl, Kriz, Mandell, Lurie, Quillen, Sullivan, To\"en, and Yuan.
    In particular, there is an $\infty$-category $\DAlg_{\bZ}^{\psi=1}$ of perfect derived $\lambda$-rings over $\bZ$ with
    trivializations of the Adams operations $\psi^p$ for all $p$ such that the functor
    $\Sscr^\op\rightarrow\DAlg_{\bZ}^{\psi=1}$ sending a space $X$ to $\bZ^X$, which models
    $\bZ$-cochains on $X$, is fully faithful on a large class of nilpotent spaces.
    Our theorem is closely related to recent work of Horel and Kubrak--Shuklin--Zakharov.
    Finally, we answer two questions of Yuan on spherical cochains.
\end{abstract}

\tableofcontents

\section{Introduction}

The starting point of this paper is a new proof of the following theorem~\cite[Thm.~2.1]{bsy}.

\begin{theorem}[Burklund--Schlank--Yuan, Lurie,~\ref{prop:homology_of_sw}]\label{thm:intromain}
    There is an adjunction
    $$\bS\bW\colon\CAlg_{\bF_p}^\perf\rightleftarrows\CAlg_{\bS_p}^{\bE_\infty,\wedge}\rcolon(-)^\flat$$
    between the category of perfect $\bF_p$-algebras and the $\infty$-category of $p$-complete $\bE_\infty$-rings where the
    right adjoint is given by $R\mapsto R^\flat=(\pi_0(R)/p)^\perf$ and the left adjoint is given by
    the functor of spherical Witt vectors.
\end{theorem}

If $k$ is a perfect $\bF_p$-algebra, then $\bS\bW(k)$ is a $p$-complete flat $\bE_\infty$-ring over
$\bS_p$ with the property that there are natural equivalences
$\bS\bW(k)\otimes_{\bS_p}\bF_p\we k$ and
$(\bS\bW(k)\otimes_{\bS_p}\bZ_p)_p^\wedge\we\bW(k)$, the $p$-typical ring of Witt vectors of $k$. In this sense, the spherical Witt vectors provide a
canonical lift to the sphere spectrum of the rings of $p$-typical Witt vectors of perfect $\bF_p$-algebras.

The approach in~\cite{bsy} to Theorem~\ref{thm:intromain} is to appeal to the deformation theory of
$\bE_\infty$-rings as in~\cite{lurie-elliptic-2}. Our approach is instead via transmutation, which
in this case amounts to constructing a specific $p$-complete $\bE_\infty$-ring, namely
$\bS[t^{1/p^\infty}]_p^\wedge$, which we show corepresents $(-)^\flat$. The existence of the left
adjoint $\bS\bW$ follows from the adjoint functor theorem, and the requisite property of lifting the
spherical Witt vectors is a computation. (Our first ``proof'' of this theorem was incorrect, as
pointed out by Maxime Ramzi and Mura Yakerson; this error is corrected in the present version of the
document.)

After explaining this result, Allen Yuan asked about possible nonconnective extensions. The
remainder of this paper is an exploration of that theme. For this, we will heavily use the theory of
derived commutative rings due to Brantner, Bhatt, Mathew, and Raksit
(see~\cite{brantner-mathew,brantner-campos-nuiten,raksit}).

We give a synthetic extension of the nonconnective Witt vector result which applies to
all $k\in\DAlg_{\bF_p}^\perf$, the $\infty$-category of perfect derived commutative $\bF_p$-algebras.
For these, there is a canonical $p$-complete derived commutative $\bZ_p$-algebra $\bW(k)$ such that
$\bW(k)\otimes_{\bZ_p}\bF_p\we k$ constructed by Holeman in~\cite{holeman-derived}.
For the moment, let $\bS_\syn$ denote the $p$-complete
synthetic sphere spectrum and let $\FDhat(\bS_\syn)_p^\wedge$ denote the $\infty$-category of complete and
$p$-complete synthetic spectra (for which, see~\cite{gikr,pstragowski-synthetic}).
Given a commutative $\bZ_p$-algebra $R$, there is an $\bE_\infty$-algebra $\ins^0 R$, which, as a
filtered spectrum, has $\F^i\we 0$ for $i>0$ and $\F^i\we R$ for $i\leq 0$ with identity transition
maps in non-positive degrees.

\begin{theorem}[Synthetic nonconnective Witt vectors,~\ref{thm:nonconnective}]
    There is a colimit-preserving functor
    $$\bS\bW_\syn\colon\DAlg_{\bF_p}^\perf\rightarrow\CAlg(\FDhat(\bS_{\syn})_p^\wedge)$$
    with the property that there are natural equivalences
    $\bS\bW_\syn(k)\otimes_{\bS_\syn}\ins^0\bF_p\we\ins^0k$ and
    $(\bS\bW_\syn(k)\otimes_{\bS_\syn}\ins^0\bZ_p)_p^\wedge\we\ins^0\bW(k)$ for 
    $k\in\DAlg_{\bF_p}^\perf$.
\end{theorem}

By taking the underlying object of the filtration we obtain a nonconnective spherical Witt vector
functor $\bS\bW\colon\DAlg_{\bF_p}^\perf\rightarrow\CAlg_{\bS_p}^{\bE_\infty,\wedge}$, however the
resulting rings $\bS\bW(k)$ do not in general give a flat deformation of $k$ to
$\bS_p$. A counterexample is given by $\bS\bW(\bF_p^{K(\bF_p,1)})$, the nonconnective spherical Witt
vectors of the $\bF_p$-cochains of the Eilenberg--Mac Lane space $K(\bF_p,1)$.

With these nonconnective Witt vectors in hand, we were motivated to consider the integral cochain (or
schematization)
problem of Grothendieck (see~\cite[Chap.~VI]{grothendieck-pursuing}).
This has recently been achieved independently by Horel in~\cite{horel} and by
Kubrak--Shuklin--Zakharov in~\cite{ksz}. Horel proves that for finite
type nilpotent spaces $X$ one can recover $X$ as $\Map(\bZ^X,\bZ)$, where the mapping space is in
the $\infty$-category associated to a model category structure on cosimplicial binomial rings (to be
defined below). Kubrak--Shuklin--Zakharov prove the analogous theorem but with target an
$\infty$-categories of derived binomial rings.

We give an a priori different approach to
a similar theorem, where we work entirely within the world of derived
commutative rings and use Sullivan's arithmetic fracture squares of
spaces~\cite{sullivan-infinitesimal}, which says
that for nilpotent spaces $X$ the natural commutative diagram
$$\xymatrix{
    X\ar[r]\ar[d]&X_\bQ\ar[d]\\
    \prod_p X_p^\wedge\ar[r]&\left(\prod_p X_p^\wedge\right)_{\bQ}
}$$ is a pullback.

Sullivan's theory of rational cdgas~\cite{sullivan-infinitesimal}
shows that one can recover the rational homotopy type $X_\bQ$ of a
nilpotent space $X$ from $\bQ^X$ as a cdga, or, for us, as a rational derived commutative ring.
Mandell~\cite{mandell-padic} also studied the analogous problem of reconstructing $X_p^\wedge$ from its $\bE_\infty$-ring
of $\bF_p$-cochains, $\bF_p^X$. It is then clear that what has been missing is a way to glue in the
pieces because the rationalization of the $\bF_p$-cochains are not interesting.

Various attempts have been made to extend these results to integral cochains. The closest to our
work, besides~\cite{horel,ksz}, is that of Ekedahl~\cite{ekedahl-minimal} who also uses binomial rings to model homotopy
types. Mandell shows in~\cite{mandell-cochains} that $\Sscr^\op\rightarrow\CAlg_{\bZ}^{\bE_\infty}$
is at least faithful on finite type nilpotent spaces, and To\"en proves the analogous theorem in the
context of affine stacks~\cite{toen-schematisation}. The problem in both To\"en and Mandell's approaches
is that there are too many maps of the relevant kinds of rings. One needs to impose additional structural
constraints in order to cut these down to size.

We achieve our integral model in three steps. The first step is to introduce the $\infty$-category of
$p$-Boolean derived
commutative $\bF_p$-algebras,
$\DAlg_{\bF_p}^{\varphi=1}$, which consists of derived commutative $\bF_p$-algebras equipped with a
trivialization $\varphi\we\id$ of the Frobenius in a strong sense. The name is motivated by Boolean rings:
commutative $\bF_2$-algebras where $x^2=x$ for all $x$, or, equivalently, where the Frobenius acts
as the identity.
We show that for finite type nilpotent spaces $X$,
one has $\Map_{\DAlg_{\bF_p}^{\varphi=1}}(\bF_p^X,\bF_p)\we X_p^\wedge$.
More generally, we establish the following theorem.

\begin{theorem}[\ref{thm:derivedstone}]
    The map $\Pro(\Sscr_{p\fin})^\op\rightarrow\DAlg_{\bF_p}^{\varphi=1}$ given by taking continuous
    $\bF_p$-cochains is an equivalence.
\end{theorem}

This is closely related to a theorem
of Kriz~\cite{kriz} who studies cosimplicial $p$-Boolean rings and gives a similar
result.

Note that as in Kriz~\cite{kriz} and Mandell~\cite{mandell-padic}, it is not enough to work with
perfect derived commutative $\bF_p$-algebras. There are functors
$$\DAlg_{\bF_p}^{\varphi=1}\rightarrow\DAlg_{\bF_p}^\perf\rightarrow\DAlg_{\bFbar_p}^\perf,$$
where the first is given by forgetting the trivialization of Frobenius and the second is extension
of scalars. The composite is an equivalence. So, the theorem above can, as in Mandell's paper, be
restated in terms of $\bFbar_p$-cochains. However, by adjunction
$$\Map_{\DAlg_{\bF_p}^\perf}(\bF_p^X,\bF_p)\we\Map_{\DAlg_{\bF_p}^{\varphi=1}}(\bF_p^X,\bF_p^{\varphi=1})\we\Map_{\DAlg_{\bF_p}^{\varphi=1}}(\bF_p^X,\bF_p^{S^1}),$$
which is equivalent to $\Map(S^1,X_p^\wedge)$, the $p$-adic free loopspace.

The second step is to give a $p$-adic lift of this category, namely $\DAlg_{\bZ_p}^{\delta,\varphi=1,\wedge}$,
the $\infty$-category of $p$-complete derived $\delta$-rings with a trivialization of the Frobenius.
This $\infty$-category is in fact equivalent to $\DAlg_{\bF_p}^{\varphi=1}$ and hence
$\Map_{\DAlg_{\bZ_p}^{\delta,\varphi=1,\wedge}}(\bZ_p^X,\bZ_p)$ recovers $X_p^\wedge$ for a finite
type nilpotent space $X$. This gives us the integral $p$-adic data we need to glue.

The third step is to define an appropriate integral model. The theory of $\lambda$-rings gives an
integral generalization of $\delta$-rings: the Adams operations $\psi^p$ define $\delta$-ring
structures for each prime $p$, and the Adams operations $\psi^n$ commute with $\psi^p$ so they act
through $\delta$-ring endomorphisms (with respect to any $p$). The analog of requiring $\varphi=1$
is to ask for $\psi^p=1$ for all $p$ simultaneously, in a way which is compatible across different
$p$.\footnote{This perspective was originally motivated by the approach of Yuan to an integral
spherical model in~\cite{yuan-integral}. We also thank Thomas Nikolaus for emphasizing the
importance of the compatibility across different primes.} For discrete $\lambda$-rings, these are the binomial rings, so named because for a
$\lambda$-ring where all Adams operations are the identity, the $\lambda$-ring operations are given
by $$\lambda^n(x)=\binom{x}{n}=\frac{x(x-1)\cdots(x-n+1)}{n!}.$$
Conversely, a torsion free commutative ring $R$ which is closed under the operations
$x\mapsto\binom{x}{n}$ for all $n\geq 0$ is canonically a $\lambda$-ring. See~\cite{elliott}
or~\cite{xantcha} for details.

We define an $\infty$-category of derived $\lambda$-rings $\DAlg_{\bZ}^\lambda$ by deriving the free
$\lambda$-ring monad, as in~\cite{raksit} in the case of derived commutative rings
or~\cite{holeman-derived}
in the case of derived $\delta$-rings.
We show that there is an action of $\B\bZ_{>0}$ on this $\infty$-category, where $\bZ_{>0}$ is the
monoid of positive integers under multiplication, and we define the
$\infty$-category of binomial derived $\lambda$-rings $\DAlg_{\bZ}^{\lambda,\psi=1}$
as $$(\DAlg_{\bZ}^\lambda)^{\h\B\bZ_{>0}}.$$
This $\infty$-category has all limits and initial object $\bZ$. Thus, for any space
$X\in\Sscr$, we can form $\bZ^X$ as a binomial derived $\lambda$-ring.
Our take on Grothendieck's problem is the following (see also~\cite{horel,ksz}).

\begin{theorem}[\ref{thm:integral_cochain}]
    The functor $\Sscr_{\bZ\ft}^\op\rightarrow\DAlg_{\bZ}^{\lambda,\psi=1}$ is fully faithful, where
    $\Sscr_{\bZ\ft}$ is the $\infty$-category of finite type nilpotent spaces.
\end{theorem}

This theorem was conceived and proved independently of the works of Horel~\cite{horel} and
Kubrak--Shuklin--Zakharov~\cite{ksz} and in fact all three theorems are distinct as they use
different $\infty$-categories of algebras to model integral homotopy types. In future work, we plan
to investigate the relationships between these models. We conjecture here that they are equivalent.

We conclude the paper by returning to spherical models. First, using synthetic cochains, we give a
$p$-adic model over the sphere in Theorem~\ref{thm:synthetic} which works for finite type $p$-complete spaces,
answering~\cite[7.4(1)]{yuan-integral}. Second, we make a remark about another question of
Yuan in Remark~\ref{rem:tcs}. Third, we construct in Proposition~\ref{prop:spherical_binomial} spherical analogs of perfect
$\lambda$-rings, which in Conjecture~\ref{conj:synthetic_binomial} we suggest should give
spherical integral models of spaces beyond the finite complex case of~\cite{yuan-integral}.

\paragraph{Relation to other work.} Besides the cochain theorems cited above, there has been work on
using various structures on chains to model homotopy types.
Quillen~\cite{quillen-rational} uses dg Lie algebras over $\bQ$, there is work of
Goerss~\cite{goerss-padic} in the $p$-adic case using coalgebras, and work of
Rivera--Wierstra--Zeinalian~\cite{rivera-wierstra-zeinalian} and
Blomquist--Harper~\cite{blomquist-harper} using coalgebras over $\bZ$. It remains 
an interesting open problem to give a Lie-theoretic approach closer in spirit to Quillen's in the
$p$-adic or integral settings, although see the recent preprint~\cite{lucio}.

\paragraph{Notation.}
If $k$ is an $\bE_\infty$-ring, let $\CAlg_k^{\bE_\infty}$ denote the $\infty$-category of
$\bE_\infty$-algebras under $k$, let $\CAlg_k^\cn$ denote the $\infty$-category of connective
$\bE_\infty$-rings under $k$, and let $\CAlg_k^{\ccn}$ denote the $\infty$-category of coconnective
$\bE_\infty$-rings under $k$. If $R$ is a commutative ring, let $\CAlg_R$ be the category of
commutative $R$-algebras. If $S$ is a (perfect) commutative $\bF_p$-algebra, let $\CAlg^\perf_{S}$ be
the category of perfect commutative $S$-algebras. We denote $\infty$-categories of
$p$-complete objects as, for example, $\CAlg_k^{\bE_\infty,\wedge}$. Other ($\infty$-)-categories of algebras will be
introduced throughout the course of the paper. We write $\PreStk_k$ for the $\infty$-category of functors
$\CAlg_k^\cn\rightarrow\Sscr$, and let $\Stk_k\subseteq\PreStk_k$ be the full subcategory of
functors satisfying flat descent. We write $\bZ_p^\un$ for $\bW(\bFbar_p)$ and $\bS_p^\un$ for the
spherical Witt vectors of $\bFbar_p$.

\paragraph{Acknowledgments.}
We would like to thank Bhargav Bhatt, Lukas Brantner, David Gepner, Paul Goerss, Adam Holeman, Geoffroy Horel,
Dmitry Kubrak, Anna Lipman, Akhil Mathew, Thomas Nikolaus, Joost Nuiten, and Allen Yuan for
a variety of helpful discussions. Special thanks go to Adam Holeman and Dmitry Kubrak for their comments on
drafts of this paper. Even more special thanks go to Maxime Ramzi and Mura Yakerson for pointing
out a critical flaw in our original attempt to prove Theorem~\ref{thm:intromain} via
Lemma~\ref{lem:main_lemma} and for discussions around finding a correct proof.

This work was supported by NSF grants
DMS-2102010 ({\em Cyclotomic spectra and $p$-divisible groups}) and DMS-2152235 ({\em FRG: higher
categorical structures in algebraic geometry}) and by Simons
Fellowship 666565 ({\em Motivic filtrations}).

\section{Spherical Witt vectors via transmutation}\label{sec:spherical_witt_vectors}

In this section, we give an alternative approach to the $p$-typical spherical Witt vector functor
constructed in~\cite{dag13,bsy}.

\subsection{Transmutation}\label{sub:transmutation}

To turn stacks over one base commutative ring into stacks over another, one can sometimes transmute,
following~\cite[Rem.~2.3.8]{bhatt-fgauges}.

\begin{definition}[Transmutation]
    Let $k$ and $R$ be $\bE_\infty$-rings. A connective $\bE_\infty$-$k$-algebra stack $A$ over $\Spec R$ is a
    functor $\CAlg_R^\cn\rightarrow\CAlg_k^\cn$ which satisfies flat descent. Given such an $A$, there is an induced
    functor $\PreStk_k\rightarrow\PreStk_R$ obtained by sending $X\in\PreStk_k$ to the functor
    $$X^A\colon B\mapsto X(A(B))$$
    on $\CAlg_R^\cn$. The prestack $X^A$ is the transmutation of $X$ by $A$.
\end{definition}

\begin{example}[The de Rham stack]
    Let $k$ be a characteristic $0$ field and let $X$ be a stack over $\Spec k$. The de Rham stack
    of $X$ is the functor $X^\dR(B)=X(B_\red)$, where $B_\red$ denotes the quotient of $B$ by its
    nilradical. If $\bA^{1,\dR}$ denotes the functor $B\mapsto B_\red$, which is also the de Rham stack
    of $\bA^1$, then $X^\dR$ is the transmutation of $X$ by $\bA^{1,\dR}$.
\end{example}

Suppose that $A=\Spec\Oscr(A)=\Map_{R}(\Oscr(A),-)$, where $\Oscr(A)$ is a possibly
nonconnective $\bE_\infty$-$k$-algebra. (Such stacks are called affine~\cite{toen-affines} or coaffine in the
literature~\cite{dag8} in related contexts.) In particular, the $\bE_\infty$-$k$-algebra structure on $A$ arises from
an $\bE_\infty$-$k$-coalgebra structure on $\Oscr(A)$. Then,
$$\CAlg_k^\cn\leftarrow\CAlg_R^\cn\rcolon A$$ preserves all limits as it is given by
$\Map_R(\Oscr(A),-)$. Thus, there is a left adjoint $\T^A$ to $A$. By definition, for a connective
$\bE_\infty$-$k$-algebra $C$ and a connective $\bE_\infty$-$R$-algebra $B$, we have
$$\Map_R(\T^A(C),B)\we\Map_k(C,\Map_R(\Oscr(A),B))\we (\Spec C)^A(B).$$
The transmutation $(\Spec C)^A$ of $\Spec C$ by $A$ is the affine scheme associated
to the connective $\bE_\infty$-$R$-algebra $\T^A(C)$.

\begin{example}[$A$ is the transmutation of the affine line]
    If $k\{t\}$ denotes the free $\bE_\infty$-$k$-algebra on an element in degree $0$, then
    $\T^A(k\{t\})\we A$ as stacks.
\end{example}

\begin{variant}
    There are variants of transmutation where one considers for example animated commutative
    rings and derived commutative rings instead of connective $\bE_\infty$-rings and all
    $\bE_\infty$-rings.
\end{variant}

\subsection{The perfect $p$-complete affine line over the sphere}\label{sec:swconnective}

Each $\bE_\infty$-space comes equipped with a canonical ``multiplication-by-$p$'' endomorphism which
we call a Frobenius map.

\begin{definition}[Frobenius on $\bE_\infty$-spaces]\label{def:frobenius}
    There is a natural transformation of the identity functor on $\bE_\infty$-spaces given for an
    $\bE_\infty$-space $M$ as the
    composition $\varphi_p\colon M\xrightarrow{\Delta_p}M^{\times p}\xrightarrow{\mu_p} M$, where the left-hand map is the
    $p$-fold diagonal and the right-hand map is the $p$-fold multiplication map induced by the $\bE_\infty$-space
    structure on $M$. 
\end{definition}

\begin{construction}[$p$-telescope]
    Given any $\bE_\infty$-space $M$, we can form the colimit $$M[1/p]=\colim_{\varphi_p}
    M=\colim(M\xrightarrow{\varphi_p} M\xrightarrow{\varphi_p}M\rightarrow\cdots).$$
    Similarly, we let $M^{[1/p]}=\lim_{\varphi_p}(\cdots\rightarrow
    M\xrightarrow{\varphi_p}M\xrightarrow{\varphi_p}M)$.
\end{construction}

\begin{definition}[$p$-perfect $\bE_\infty$-spaces]\label{def:pperfect}
    An $\bE_\infty$-space $M$ is $p$-perfect if $\varphi_p$ is an equivalence.
    The inclusion of $p$-perfect $\bE_\infty$-spaces into all $\bE_\infty$-spaces admits a left
    adjoint, $M\mapsto M_\perf$.
\end{definition}

\begin{warning}[Telescopes are not perfect]
    In general, the natural map $M[1/p]\rightarrow M_\perf$ is not an equivalence, as was
    erroneously claimed in the first version of this paper.
    See Remark~\ref{rem:pitfalls} for details.
\end{warning}

%

Despite the warning above, telescopes are enough in certain $p$-complete contexts.

\begin{proposition}\label{prop:discrete}
    If $R$ is a $p$-complete $\bE_\infty$-ring, then $(\Omega^{\infty}R)^{[1/p]}$ is discrete, where we compute
    the inverse limit tower using the $\bE_\infty$-structure on the infinite loopspace arising
    from multiplication. Moreover, $\pi_0(\Omega^{\infty}R)^{[1/p]}\iso(R/p)^{\perf}$, the inverse
    limit perfection of the commutative $\bF_p$-algebra $R/p$.
\end{proposition}

\begin{proof}
    More generally, suppose that $M$ is an $\bE_\infty$-space satisfying the following conditions:
    \begin{enumerate}
        \item[(a)] $M$ is simple as a space;
        \item[(b)] for all $x\in M$ and all $i\geq 1$, $\pi_i(M,x)$ is a derived $p$-complete abelian group.
    \end{enumerate}
    Then, $M^{[1/p]}$ is discrete.

    To prove the claim, consider a point $x\in M$ and the composition
    $$\pi_i(M,x)\xrightarrow{\Delta_p}\pi_i(M^{\times p},(x,\ldots,x))\xrightarrow{\mu_p}\pi_i(M,x^p).$$
    In general, while the map $\mu_p\colon M^{\times p}\rightarrow M$ is $\Sigma_p$-equivariant for the trivial
    action on $M$, the induced map on connected components composition $M^{\times
    p}_{(x,\ldots,x)}\rightarrow M_{x^p}$ need not be
    $\Sigma_p$-equivariant as a map of pointed spaces, though it will be as a map of unpointed
    spaces. Here, if $x\in M$, then $M_x$ denotes the component of $M$
    containing $x$. Fix $a\in\pi_i(M,x)$. Then,
    $\Delta_p(a)=(a,0,0,\ldots,0)+(0,a,0,\ldots,0)+\cdots+(0,0,0,\ldots,a)\in\pi_i(M^{\times
    p},(x,\ldots x))\iso(\pi_i(M,x))^{\times p}$. Therefore,
    $\mu_p(\Delta_p(a))=\mu_p(a,0,0,\ldots,0)+\cdots+\mu_p(0,0,0,\ldots,a)$.
    Write $a_i$ for the element with $a$ in the $i$th spot and zeros elsewhere. Thus,
    $\Delta_p(a)=a_1+\cdots+a_p$ and $\mu_p(\Delta_p(a))=\mu_p(a_1)+\cdots+\mu_p(a_p)$.
    By simplicity of $M$, if $(i\,j)\in\Sigma_p$ is the action permutation which switches $i$ and
    $j$, then $\mu_p(a_i)=\mu_p(a_j)$, because the unbased homotopy between
    $\mu_p\circ(i\, j)$ and $\mu_p$ can be upgraded to a based homotopy. If $r$ denotes the idempotent of $M^{\times
    p}$ given by projecting onto the first component and sending all other components to the
    basepoint $x$, then $\mu_p(\Delta_p(a))=p\mu_p(r(\Delta_p(a)))$. It follows that on homotopy groups,
    $\varphi_{p,*}\colon\pi_i(M,x)\rightarrow\pi_i(M,x^p)$ is divisible as a function by $p$ for $i\geq 1$.
    Thus, the positive homotopy groups of the inverse limit vanish at all basepoints by derived
    $p$-completeness.

    The final claim is precisely~\cite[Lem.~3.4(i)]{scholze-perfectoid}.
\end{proof}

If $M$ is an $\bE_\infty$-space, let $\bS[M]$ denote the spherical monoid ring and
$\bS[M]_p^\wedge$ its $p$-completion. Let $\Sscr^{\bE_\infty}$ denote the $\infty$-category of
$\bE_\infty$-spaces.

\begin{corollary}\label{cor:collapse}
    If $M$ is an $\bE_\infty$-space and $M_\perf$ is its $p$-perfection, then the natural maps
    $\bS[M[1/p]]_p^\wedge\rightarrow\bS[M_\perf]_p^\wedge\rightarrow\bS[(\pi_0M)[1/p]]_p^\wedge$
    are equivalences.
\end{corollary}

\begin{proof}
    The factorization of $M[1/p]\rightarrow(\pi_0M)[1/p]$ through $M[1/p]\rightarrow M_\perf$
    exists because $M[1/p]\rightarrow M_\perf$ induces an isomorphism on
    $\pi_0$. Indeed, in discrete commutative monoids, the perfection is computed via the telescope.
    For any $p$-complete $\bE_\infty$-ring $R$, we have $\Map_{\CAlg_\bS^{\bE_\infty}}(\bS[N],R)\we\Map_{\Sscr^{\bE_\infty}}(N,(\Omega^\infty R)^{[1/p]})$ for $N\in\{M[1/p],M_\perf,(\pi_0M)_\perf\}$.
    As $(\Omega^\infty R)^{[1/p]}$ is discrete by Proposition~\ref{prop:discrete} and as the maps $M[1/p]\rightarrow
    M_\perf\rightarrow(\pi_0M)[1/p]$ induce isomorphisms on $\pi_0$, the result follows.
\end{proof}

\begin{remark}
    In particular, if $M$ is $p$-perfect, then $\bS[M]_p^\wedge\rightarrow\bS[\pi_0M]_p^\wedge$ is
    an equivalence.
\end{remark}

\begin{remark}[Pitfalls]\label{rem:pitfalls}
    In the previous version of this paper, I made the incorrect claim that $M[1/p]$ gives a
    construction of the $p$-perfection of an $\bE_\infty$-monoid. This does not work, as pointed
    out to me by Maxime Ramzi and Mura Yakerson. They suggested
    considering the example of $\Fin^\we$, the groupoid of finite sets and bijections.
    The multiplication-by-$2$ map gives, on components, maps
    $\B\Sigma_n\rightarrow\B\Sigma_n\times\B\Sigma_n\rightarrow\B\Sigma_{2n}$ sending a cycle
    $\sigma$ to the ``block cycle'' $$\begin{pmatrix}\sigma&0\\0&\sigma\end{pmatrix}.$$
    This assignment is not equivalent to
    $$\sigma\mapsto\begin{pmatrix}\sigma^2&0\\0&1\end{pmatrix},$$ which is what one wants to
    say in order to witness the map being the square of another map. However,
    $\begin{pmatrix}\sigma&0\\0&1\end{pmatrix}$ is conjugate to
    $\begin{pmatrix}1&0\\0&\sigma\end{pmatrix}$ in $\Sigma_{2n}$, so
    the issue disappears in the case of $\Omega^\infty R$ as it has, in particular,
    abelian fundamental groups.
\end{remark}

\begin{example}
    Let $F$ be the free $\bE_\infty$-space on a point and let $F[1/p]$ be its $p$-telescope. Note that
    $\pi_0 F[1/p]\we\bZ[\tfrac{1}{p}]_{\geq 0}$. The natural map
    $\bS[F[1/p]]_p^\wedge\rightarrow\bS[\bZ[\tfrac{1}{p}]_{\geq 0}]_p^\wedge$ is an equivalence by
    Corollary~\ref{cor:collapse}.
    We write $\bS[F[1/p]]$ as $\bS\{t^{1/p^\infty}\}$ and $\bS[\bZ[\tfrac{1}{p}]_{\geq 0}]$ as
    $\bS[t^{1/p^\infty}]$.
\end{example}

\begin{corollary}\label{cor:cotangent}
    For any $\bE_\infty$-ring, the $\bE_\infty$-cotangent complex $\L_{R[t^{1/p^\infty}]/R}^{\bE_\infty}$
    vanishes after $p$-completion.
\end{corollary}

\begin{proof}
    We can reduce by base change for the cotangent complex to the case of $R$ connective, in which case
    the cotangent complex of $R\{t^{1/p^\infty}\}$ has Tor-amplitude in $[0,0]$. On the other hand,
    $$\pi_0\left(\L_{R[t^{1/p^\infty}]/R}^{\bE_\infty}/p\right)\iso\Omega^1_{((\pi_0(R)/p)[t^{1/p^\infty}])/(\pi_0(R)/p)},$$
    which vanishes by relative perfectness.
\end{proof}

\begin{lemma}\label{lem:main_lemma}
    The $p$-complete $\bE_\infty$-algebra $\bS\{t^{1/p^\infty}\}_p^\wedge\we\bS[t^{1/p^\infty}]_p^\wedge$ corepresents a
    limit-preserving functor
    $$\CAlg_{\bF_p}^\perf\leftarrow\CAlg_{\bS}^{\cn,\wedge}\rcolon\bA^{1,\perf,\bS\bW}$$
    from connective $p$-complete $\bE_\infty$-algebras to the $1$-category of perfect
    $\bF_p$-algebras given by the formula $\bA^{1,\perf,\bS\bW}(B)=(\pi_0(B)/p)^\perf$.
\end{lemma}

\begin{proof}
    Using the $p$-adic equivalence $\bS\{t^{1/p^\infty}\}_p^\wedge\we\bS[t^{1/p^\infty}]_p^\wedge$
    we see that $\Map(\bS[t^{1/p^\infty}]_p^\wedge,B)\we(\Omega^\infty B)^{[1/p]}$. This
    was computed in Proposition~\ref{prop:discrete}.
    Limit preservation
    is automatic as we map out of a $p$-complete $\bE_\infty$-algebra.
\end{proof}

\begin{notation}
    We also write $(-)^\flat$ for $\bA^{1,\perf,\bS\bW}$.
\end{notation}

\begin{definition}[Spherical Witt vectors]
    The spherical Witt vectors functor
    $\bS\bW\colon\CAlg_{\bF_p}^\perf\rightarrow\CAlg_{\bS}^{\cn,\wedge}$ is the left adjoint of the
    functor $\bA^{1,\perf,\bS\bW}$. By construction, the ring of functions on the transmutation of $\Spec C$
    with respect to $\bA^{1,\perf,\bS\bW}$ is $\bS\bW(C)$. More generally, if $X$ is a prestack on
    $\CAlg_{\bF_p}^\perf$, we denote by $X^{\bS\bW}$ its transmutation, which is a prestack on
    $\CAlg_{\bS_p}^{\cn,\wedge}$.
\end{definition}

\begin{remark}
    We could just as well consider $\bS\bW$ as taking values in $\CAlg_{\bS}^\wedge$, but no
    additional generality is gained as the natural map $(\tau_{\geq 0}R)^\flat\rightarrow R^\flat$
    is an equivalence for all $R\in\CAlg_{\bS}^\wedge$.
\end{remark}

To justify the terminology, we verify the following facts about $\bS\bW$.

\begin{proposition}\label{prop:homology_of_sw}
    Let $k$ be a perfect $\bF_p$-algebra and let $\bS\bW(k)$ be the spherical Witt vectors of $k$.
    \begin{enumerate}
        \item[{\em (a)}] The $\bE_\infty$-ring $\bS\bW(k)$ is $p$-completely flat over $\bS_p$.
        \item[{\em (b)}] There is a natural equivalence $\pi_0\bS\bW(k)\iso\bW(k)$, where $\bW(k)$
            denotes the ring of $p$-typical Witt vectors.
        \item[{\em (c)}] There is a natural equivalence $\bS\bW(k)\otimes_{\bS_p}\bF_p\we k$.
        \item[{\em (d)}] There is a natural equivalence $(\bS\bW(k)\otimes_{\bS_p}\bZ_p)_p^\wedge\we\bW(k)$.
    \end{enumerate}
\end{proposition}

First, we need a lemma.

\begin{lemma}\label{lem:animated}
    The $1$-category $\CAlg_{\bF_p}^\perf$ is the animation of its full subcategory consisting of perfections of
    finitely presented polynomial $\bF_p$-algebras.
\end{lemma}

\begin{proof}
    Let $\DAlg_{\bF_p}^\cn$ denote the $\infty$-category of animated commutative $\bF_p$-algebras and
    let $\DAlg_{\bF_p}^{\cn,\perf}\subseteq\DAlg_{\bF_p}^\cn$ be the full subcategory of animated
    commutative $\bF_p$-algebras where Frobenius is an equivalence. It is a standard fact that
    $\pi_0\colon\DAlg_{\bF_p}^{\cn,\perf}\rightarrow\CAlg_{\bF_p}^\perf$ is an equivalence;
    see~\cite[Prop.~11.6]{bhatt-scholze-witt}.\footnote{For example, consider $\Omega^\infty R$
    as an $\bE_\infty$-space with respect to multiplication. The Frobenius on $R$ induces an
    endomorphism of $\Omega^\infty R$ which agrees (by comparison between simplicial commutative
    rings and simplicial commutative monoids) with the $\bE_\infty$-space Frobenius
    $\varphi_p$ of Definition~\ref{def:pperfect}. For a perfect animated commutative
    $\bF_p$-algebra, this endomorphism is an equivalence on the one hand, but it is zero
    in positive degrees on the other. Thus, $\pi_iR=0$ for $i>0$.} The resulting forgetful
    functor $\CAlg_{\bF_p}^\perf\rightarrow\DAlg_{\bF_p}^\cn$ preserves all limits and also all sifted
    colimits. But, it also preserves all coproducts by~\cite[Lem.~3.16]{bhatt-scholze-witt}, so the
    forgetful functor preserves all colimits. The
    left and right adjoints are given by the colimit and limit perfections, $(-)_\perf$ and
    $(-)^\perf$, respectively. In
    particular, it follows that $(-)_\perf\colon\DAlg_{\bF_p}^\cn\rightarrow\CAlg_{\bF_p}^\perf$
    preserves compact projective objects (and is a localization).
    Denote by $\Cscr\subseteq\CAlg_{\bF_p}^\perf$ the full subcategory of $\CAlg_{\bF_p}^\perf$
    consisting of perfections of finitely presented polynomial $\bF_p$-algebras and consider the
    functor $\Pscr_\Sigma(\Cscr)\rightarrow\CAlg_{\bF_p}^\perf$ induced by inclusion of $\Cscr$ into
    $\CAlg_{\bF_p}^\perf$. As
    $\Cscr\rightarrow\CAlg_{\bF_p}^\perf$ is fully faithful and the essential image consists of
    compact projective objects, it follows by~\cite[Prop.~5.5.8.22]{htt} that
    $\Pscr_\Sigma(\Cscr)\rightarrow\CAlg_{\bF_p}^\perf$ is fully faithful. It is also essentially
    surjective: by taking a perfect commutative $\bF_p$-algebra $R$ and a simplicial commutative
    polynomial ring $S_\bullet$ such that $|S_\bullet|\we R$, it follows that
    $S_{\bullet,\perf}$ is simplicial perfect polynomial and $|S_{\bullet,\perf}|\we R$.
\end{proof}

\begin{remark}
    The motto here is that animated perfect commutative $\bF_p$-algebras are perfect animated
    commutative $\bF_p$-algebras are perfect commutative $\bF_p$-algebras.
    In other familiar examples, like animated abelian groups or animated commutative rings, the
    animations are $\infty$-categories which are not $1$-categories: the mapping spaces are not
    $0$-truncated. For the animation to be discrete is a very special occurrence which has to do
    with the fact that the finitely presented polynomial rings are compact projective in the
    $1$-category $\CAlg_{\bF_p}^\perf$.
\end{remark}

\begin{proof}[Proof of Proposition~\ref{prop:homology_of_sw}]
    Consider the composite functor
    $$\CAlg_{\bF_p}^\perf\xrightarrow{\bS\bW}\CAlg_{\bS}^{\cn,\wedge}\xrightarrow{(-)\otimes_{\bS}\bF_p}\CAlg_{\bF_p}^{\cn}.$$
    It is enough
    to show that this functor takes values in (discrete) perfect commutative $\bF_p$-algebras. The
    total right adjoint sends a connective $\bE_\infty$-$\bF_p$-algebra $B$ to $(\pi_0B)^\perf$. By
    construction, the left adjoint sends $\bF_p[t^{1/p^\infty}]$ to $\bF_p[t^{1/p^\infty}]$ and
    similarly for perfections polynomial rings over $\bF_p$ on finitely many variables.
    To conclude, it is enough by Lemma~\ref{lem:animated} to show that perfect $\bF_p$-algebras are closed under colimits
    in $\CAlg_{\bF_p}^\cn$. For this, we can use the fact that
    the forgetful functor from animated commutative $\bF_p$-algebras
    to $\bE_\infty$-algebras over $\bF_p$ preserves all colimits (see~\cite[Prop.~4.2.27]{raksit})
    as well as the fact that $\CAlg_{\bF_p}^\perf\rightarrow\DAlg_{\bF_p}^{\cn}$ preserves colimits, used
    in the proof of Lemma~\ref{lem:animated}.
\end{proof}

\begin{corollary}\label{cor:sw_fully_faithful}
    The functor $\bS\bW$ is fully faithful.
\end{corollary}

\begin{proof}
    Indeed, the unit of the adjunction is an equivalence.
\end{proof}

\begin{remark}
    Consider the full subcategory of $p$-complete connective $\bE_\infty$-rings on
    $\bS[t^{1/p^\infty}]_p^\wedge$ and all of its finite coproducts. This $\infty$-category is
    equivalent to $\CAlg_{\bF_p}^{\perf,\omega}$. Another perspective on $\bS\bW$ is that
    it is the left Kan extension of the induced inclusion
    $\CAlg_{\bF_p}^{\perf,\omega}\rightarrow\CAlg_p^{\cn,\wedge}$ to $\CAlg_{\bF_p}^\perf$.
\end{remark}

\begin{corollary}\label{cor:cotangent_vanishing}
    If $C$ denotes a perfect $\bF_p$-algebra, then
    \begin{enumerate}
        \item[{\em (a)}] $\L_{C/\bF_p}^{\bE_\infty}\we 0$ and
        \item[{\em (b)}] $(\L_{\bS\bW(C)/\bS_p}^{\bE_\infty})_p^\wedge\we 0$.
    \end{enumerate}
\end{corollary}

\begin{proof}
    The corollary holds for $C\iso\bF_p[t^{1/p^\infty}]$ by Corollary~\ref{cor:cotangent} and hence for all
    $C$ by closing up under sifted colimits.
\end{proof}

\begin{corollary}[Recognition]
    A $p$-complete connective $\bE_\infty$-ring is in the essential image of $\bS\bW$ if and only if
    its $\bF_p$-homology is a discrete perfect $\bF_p$-algebra.
\end{corollary}

\begin{proof}
    Given such a $B$ we have the counit map $\bS\bW((\pi_0B/p)^\perf)\rightarrow B$. By hypothesis,
    this map induces an equivalence in $\bF_p$-homology, so it is an equivalence since $B$ is
    connective.
\end{proof}

\begin{variant}[$Q$-based Witt vectors]\label{var:qwitt}
    Viewing $\bS\bW$ as taking values in
    $\CAlg_{\bS_p}^{\bE_\infty,\wedge}$, if $Q$ is a general
    $\bE_\infty$-ring, then we can consider the composition
    $$\CAlg_{\bF_p}^\perf\xrightarrow{\bS\bW}\CAlg_{\bS_p}^{\bE_\infty,\wedge}\xrightarrow{((-)\otimes_{\bS_p}Q)_p^\wedge}\CAlg_Q^{\bE_\infty,\wedge},$$
    which has right adjoint given by the formula $R\mapsto R^\flat$.
    Write $\bW_Q$ for the composition, which we call the $Q$-based Witt vectors functor.
    For example, we have $\bS\bW\we\bW_{\bS_p}$ and $\bW\we\bW_{\bZ_p}$. The functor $\bW_Q$ takes
    values in $p$-complete $p$-completely flat $\bE_\infty$-$Q$-algebras. Moreover, if $Q\rightarrow Q'$ is a map
    of $p$-complete
    $\bE_\infty$-rings, then the natural map $(\bW_Q(k)\otimes_QQ')_p^\wedge\rightarrow\bW_{Q'}(k)$ is an
    equivalence for every $k\in\CAlg_{\bF_p}^\perf$. In particular, if $Q$ is a perfect commutative
    $\bF_p$-algebra, then $\bW_Q(k)\we k\otimes_{\bF_p}Q$. If $Q$ is a derived $p$-complete
    commutative ring, then $\bW_Q(k)\we(\bW(k)\otimes_{\bZ_p}Q)_p^\wedge$. 
\end{variant}

\begin{example}[Strict units via transmutation]\label{ex:strict}
    Carmeli has proved in~\cite{carmeli} that if $k$ is a perfect $\bF_p$-algebra, then the strict Picard
    space $\bG_\pic(\bS\bW(k))$ of $\bS\bW(k)$ fits into a canonical fiber sequence $$\Sigma
    k^\times\rightarrow\bG_\pic(\bS\bW(k))\rightarrow\Pic^\heart(k),$$ where $\Pic^\heart(k)$ denotes
    the classical Picard group of $k$, which is to say the Picard group of the symmetric monoidal
    abelian category $\Mod_k$. The natural counit map $\bG_\pic(\bS\bW((-)^\flat))\rightarrow\bG_\pic(-)$ can
    thus be rewritten as a natural transformation $\B\Gm^{\bS\bW}\rightarrow\bG_\pic$, which is an
    equivalence when restricted to the essential image of $\bS\bW$.
\end{example}

\subsection{Nonconnective spherical Witt vectors}\label{sub:nonconnective}

Allen Yuan has asked if the spherical Witt vector functor extends to $\DAlg_{\bF_p}^\perf$. We
show that it does, with some caveats; see also Section~\ref{sec:synthetic}.

\begin{definition}[Synthetic spectra]
    Let $\bS_\syn$ denote the $p$-complete synthetic sphere spectrum, which is a complete (and
    $p$-complete) filtered $\bE_\infty$-ring. This filtered spectrum has a spectral sequence whose
    $\bE_1$-page is the $\bE_2$-page of the $p$-complete $\MU$-based Adams--Novikov spectral
    sequence. It is also the even filtered $p$-complete sphere
    by~\cite{hrw} or~\cite{pstragowski-even}. Let $\FD(\bS_\syn)_p^\wedge$ denote the $\infty$-category of
    $\bS_\syn$-modules in filtered $p$-complete spectra and let $\FDhat(\bS_\syn)_p^\wedge$ denote the full
    subcategory of complete filtrations.\footnote{Beware of the two notions of completion at play here.} See~\cite{gikr,pstragowski-synthetic} for details.
\end{definition}

\begin{definition}[Synthetic spherical Witt vectors]
    If $k\in\CAlg_{\bF_p}^\perf$,
    we let $\bS\bW_\syn(k)$ be the $p$-complete even filtration on $\bS\bW(k)$. Alternatively, since
    $\bS\bW(k)$ is $p$-completely flat over $\bS_p$,
    $$\bS\bW_\syn(k)\we(\bS\bW(k)\otimes_{\bS_p}\bS_\syn)_p^\wedge.$$ This construction defines a functor
    $\bS\bW_\syn\colon\CAlg_{\bF_p}^\perf\rightarrow\CAlg(\FD(\bS_\syn)_p^\wedge)$, which
    takes values in complete $p$-complete synthetic spectra.
\end{definition}

\begin{remark}[Flat filtrations]
    Let $\F^\star R$ be a filtered $\bE_\infty$-ring. A filtered $p$-complete $\F^\star R$-module $\F^\star M$ is flat
    if the natural map $(\F^0M\otimes_{\F^0R}\F^\star R)_p^\wedge\rightarrow\F^\star M$ is an equivalence.
    As $\F^0_\syn\bS_p\we\bS_p$, each $\bS\bW_\syn(k)$ is flat as a $p$-complete synthetic spectrum.
\end{remark}

\begin{definition}[Synthetic $Q$-based Witt vectors]\label{def:synqw}
    Given an $\bE_\infty$-ring $Q$, we let
    $$\F^\star_\syn\bW_Q(k)=(\bS\bW(k)\otimes_{\bS_p}\F^\star_\ev Q)_p^\wedge\we(\bW_{\F^0_\ev
    Q}(k)\otimes_{\F^0_\ev Q}\F^\star_\ev Q)_p^\wedge,$$ which is in particular flat.
    Here, $\F^\star_\ev Q$ denotes the $p$-complete even filtration of~\cite{hrw} on $Q$.
    We will typically only apply this in cases when the even filtration on $Q$ is well behaved. In general,
    $\F^\star_\syn\bW_Q(k)\we(\bS\bW_\syn(k)\otimes_{\bS_\syn}\F^\star_\ev Q)_p^\wedge$.
    If the even filtration on $Q$ is exhaustive, so that the natural map
    $Q\rightarrow\F^{-\infty}_\ev Q$ is an equivalence, then $(\F^{-\infty}_\syn\bW_Q(k))_p^\wedge\we\bW_Q(k)$;
    that is, the filtration $\F^\star_\syn\bW_Q(k)$ is $p$-completely exhaustive.
    If $Q$ is connective, then $\F^\star_\syn\bW_Q(k)$ is complete (and exhaustive) as it is the flat filtration on a
    connective $Q$-module and because $\F^i_\ev Q$ is $i$-connective by an unpublished argument of
    Burklund and Krause. In general, we write $\F^\star_\syn\bWhat_Q(k)$ for the completion of
    $\F^\star_\syn\bW_Q(k)$ with respect to the filtration.
\end{definition}

We review Holeman's theory of non-linear right left extensions, a tool to derive functors of
finitely presented polynomial rings.

\begin{definition}
    Let $\Dscr$ be a presentable $\infty$-category. Let $\Perf(\bF_p)_{\leq 0}\subseteq\Perf(\bF_p)$
    be the full subcategory of perfect complexes with vanishing positive homology, let
    $\Vect(\bF_p)^\omega$ be the category of finite dimensional $\bF_p$-vector spaces, and let
    $\Poly(\bF_p)^\omega$ be the category of finitely presented free commutative $\bF_p$-algebras,
    viewed as a full subcategory of $\CAlg_{\bF_p}$.
    \begin{enumerate}
        \item[(i)] A functor $F\colon\Vect_{\bF_p}^\omega\rightarrow\Dscr$ is right extendable
            if the right Kan extension of $F$ to $\Perf(\bF_p)_{\leq 0}$ preserves finite
            coconnective geometric realizations (i.e., those geometric realizations which are
            $n$-skeletal for some finite $n$ and which are coconnective when computed in $\D(k)$).
        \item[(ii)] A functor $\Poly_{\bF_p}^\omega\rightarrow\Dscr$ is right extendable if its
            composition with the free commutative algebra functor $\Vect_{\bF_p}^\omega\rightarrow\Poly_{\bF_p}^\omega$
            is right extendable in the sense of part (i).
    \end{enumerate}
\end{definition}

\begin{example}
    Suppose that $\Dscr$ is equipped with a $t$-structure which is compatible with filtered colimits
    in the sense that $\Dscr_{\leq 0}$ is closed under filtered colimits in $\Dscr$. If
    $F\colon\Vect_{\bF_p}^\omega\rightarrow\Dscr_{\leq 0}$
    admits an exhaustive, increasing $\bN$-indexed filtration by additively
    polynomial subfunctors, then $F$ is right extendable when viewed as a functor to $\Dscr$. At each stage of the filtration, this
    follows from~\cite[Thm.~3.26]{brantner-mathew}, and filtered colimits commute with
    totalizations of uniformly bounded above cosimplicial objects in $\Dscr$.
    This is the situation for $\Sym_{k}$ as a functor from $\Vect_{k}^\omega$ to
    $\D(k)$ when $k$ is a field.
\end{example}

\begin{lemma}\label{lem:right_extendable}
    The functor $\Sym_{\bF_p}^\perf\colon\Vect_{\bF_p}^\omega\rightarrow\D(\bF_p)$ is right extendable.
\end{lemma}

\begin{proof}
    As an endofunctor of $\Mod_{\bF_p}$, $\Sym_{\bF_p}^\perf$ is the filtered colimit of
    $\Sym_{\bF_p}$ along Frobenius. It follows that the right Kan extension of the restriction
    of $\Sym_{\bF_p}^\perf$ from $\Vect_{\bF_p}^\omega$ to $\Perf(\bF_p)_{\leq 0}$ is the ``perfection'' of the right Kan extension of
    the restriction of $\Sym_{\bF_p}$ to $\Perf(\bF_p)_{\leq 0}$ using that filtered colimits commute with
    uniformly bounded above totalizations.
    As this filtered colimit commutes in particular with $n$-skeletal coconnective geometric
    realizations, the result follows from the right extendability of $\Sym_{\bF_p}$.
\end{proof}

We will have occasion to argue that a functor $\Poly_{\bF_p}^\omega\rightarrow\Cscr$ is right extendable by
composing with a functor $F\colon\Cscr\rightarrow\Dscr$ and making some argument. The following
lemma encodes possible sufficient conditions on $F$.

\begin{lemma}\label{lem:extendability}
    Let $S\colon\Poly_{\bF_p}^\omega\rightarrow\Cscr$ be a functor to an $\infty$-category which
    admits all colimits and let $F\colon\Cscr\rightarrow\Dscr$ be a functor to another such
    $\infty$-category. Suppose that the composition $F\circ S$ is right extendable. If $F$ is
    conservative and either
    \begin{enumerate}
        \item[{\em (a)}] $F$ commutes with all totalizations and finite geometric realizations or
        \item[{\em (b)}] $S$ takes uniformly bounded above values (for example coconnective values)
            and $F$ commutes with finite geometric realizations and with uniformly bounded above totalizations
            with respect to $t$-structures on $\Cscr$ and $\Dscr$ which are compatible with filtered
            colimits,
    \end{enumerate}
    then $S$ is right extendable.
\end{lemma}

\begin{proof}
    The lemma under assumption (a) follows by conservativity. Under (b), one uses that filtered
    colimits commute with totalizations
    of uniformly bounded objects. Indeed, if $S^R$ and $(F\circ S)^R$ denote the right Kan
    extensions from $\Vect_{\bF_p}^\omega$ to $\Perf_{\bF_p,\leq 0}$ of the compositions
    of $S$ and $F\circ S$ with $\LSym\colon\Vect_{\bF_p}^\omega\rightarrow\Poly_{\bF_p}^\omega$, then assumption (b) implies that $F\circ
    S^R\we(F\circ S)^R$. As $(F\circ S)^R$ preserve finite coconnective geometric realizations by
    hypothesis and since $F$ is conservative and commutes with finite geometric realizations, it
    follows that $S^R$ preserves finite coconnective geometric realizations as well.
\end{proof}

\begin{construction}[Non-linear right left extension]
    \label{const:nlrl}
    If $F\colon\Poly_{\bF_p}^\omega\rightarrow\Dscr$ is right extendable, its non-linear right left
    extension is constructed as follows. Let $F'$ denote the restriction of $F$ along
    $\Sym_{\bF_p}\colon\Vect_{\bF_p}^\omega\rightarrow\Poly_{\bF_p}^\omega$. Let $F'^{\R}$ denote the right Kan
    extension of $F'$ along the inclusion $\Vect_{\bF_p}^\omega\rightarrow\Perf_{\bF_p,\leq 0}$.
    Let $F'^{\R\L}$ be the left Kan extension of $F'^{\R}$ from $\Perf_{\bF_p,\leq
    0}$ to $\D(\bF_p)$.
    For $R\in\DAlg_{\bF_p}$, one considers the bar construction
    $(\LSym_{\bF_p})^{\circ(\bullet+1)}(R)$, a simplicial object whose colimit is $R$.
    Then, using that $F'$ arises from a functor of finitely presented polynomial $\bF_p$-algebras to
    begin with, Holeman shows in~\cite[Const.~2.2.17]{holeman-derived} that the pointwise
    application of $F'^{\R\L}$ is well-defined and compatible with the maps in the simplicial
    diagram and one defines
    the non-linear right left extension of $F$ on $R$ to be the geometric realization
    $$F^{\nlRL}(R)=|F'^{\R\L}((\LSym_{\bF_p})^{\circ(\bullet+1)}(R))|.$$
    This construction is natural in $R$ and gives, for any presentable $\infty$-category, a functor
    $$\Fun^{\mathrm{ext}}(\Poly_{\bF_p}^\omega,\Dscr)\xrightarrow{F\mapsto F^{\nlRL}}\Fun_\Sigma(\DAlg_{\bF_p},\Dscr).$$
    where the left-hand side is the full subcategory of right extendable functors and the
    right-hand side is the full subcategory of functors preserving sifted colimits.
\end{construction}

\begin{lemma}[Properties of non-linear right left extension]
    Let $F\colon\Poly_{\bF_p}^\omega\rightarrow\Dscr$ be right extendable and let
    $F^{\nlRL}\colon\DAlg_{\bF_p}\rightarrow\Dscr$ be the induced non-linear right left
    extension.
    \begin{enumerate}
        \item[{\em (a)}] The restriction of $F^{\nlRL}$ to $\Poly_{\bF_p}^\omega$ is naturally
            equivalent to $F$.
        \item[{\em (b)}] The restriction of $F^{\nlRL}$ to $\DAlg_{\bF_p}^\cn$ is naturally
            equivalent to the left Kan extension of $F$.
    \end{enumerate}
\end{lemma}

\begin{proof}
    See~\cite[Lem.~2.2.18, Obs.~2.2.19]{holeman-derived}.
\end{proof}

\begin{example}
    Suppose that $F\colon\Poly_{\bF_p}^\omega$ is right extendable and
    arises as the restriction of a functor
    $\widetilde{F}\colon\DAlg_{\bF_p}\rightarrow\Dscr$ to $\Poly_{\bF_p}^\omega$.
    We want to compare $F^{\nlRL}$ to $\widetilde{F}$. There is a canonical map
    $$|\widetilde{F}((\LSym_{\bF_p})^{\circ(\bullet+1)}(R))|\rightarrow\widetilde{F}(R)$$
    for $R\in\DAlg_{\bF_p}$. Write $\widetilde{F}^\L$ for the left-hand construction, so there is a map
    $\widetilde{F}^\L\rightarrow\widetilde{F}$ of functors $\DAlg_{\bF_p}\rightarrow\Dscr$.
    By construction, $\widetilde{F}^\L$ arises from an
    $\LSym_{\bF_p}$-module $\widetilde{F}^\L_\Mod$
    in functors $\D(\bF_p)\rightarrow\Dscr$. Let $\widetilde{F}^{\L\L}$
    denote the sifted colimit-preserving approximation to $\widetilde{F}^\L$, obtained by taking
    the sifted colimit-preserving approximation $\widetilde{F}^{\L\L}_\Mod$
    of $\widetilde{F}^\L_\Mod$ in $\LSym_{\bF_p}$-modules and geometrically
    realizing using the bar construction. Finally, by construction, there is a map
    $\widetilde{F}^{\L\L}_\Mod\rightarrow F'^{\R\L}$ of $\LSym_{\bF_p}$-modules (in the notation
    of Construction~\ref{const:nlrl}),
    which is obtained by restriction of the left-hand side to $\Vect_{\bF_p}^\omega$, right Kan
    extending to $\Perf_{\bF_p,\leq 0}$, and left Kan extending. 
    It follows that there is a diagram
    $$\xymatrix{
        \widetilde{F}^{\L\L}\ar[r]\ar[d]&\widetilde{F}^\L\ar[r]&\widetilde{F}\\
        F^{\nlRL}.
    }$$
    If $\widetilde{F}$ preserves sifted colimits, then the horizontal arrows above are
    equivalences, from which one obtains (by taking inverses) a canonical map
    $\widetilde{F}\rightarrow F^{\nlRL}$. This map is an equivalence if and only if the
    restriction of $\widetilde{F}$ to $\Perf_{\bF_p,\leq 0}$ along $\LSym_{\bF_p}$ is right Kan
    extended from $\Vect_{\bF_p}^\omega$.
\end{example}

\begin{definition}
    We say that a sifted-colimit preserving functor
    $\widetilde{F}\colon\DAlg_{\bF_p}\rightarrow\Dscr$ whose restriction to
    $\Poly_{\bF_p}^\omega$ is right extendable is
    non-linearly right left extended if the natural map $\widetilde{F}\rightarrow F^{\nlRL}$ is
    an equivalence.
\end{definition}

\begin{example}[The identity functor as an $\nlRL$ extension]\label{ex:identity}
    The identity functor of $\DAlg_{\bF_p}$ is non-linearly right left extended.
\end{example}

\begin{example}[Perfection as an $\nlRL$ extension]
    \label{ex:perfection}
    The perfection functor $(-)_\perf\colon\DAlg_{\bF_p}\rightarrow\DAlg_{\bF_p}^\perf$ is non-linearly right
    left extended by an argument as in Lemma~\ref{lem:right_extendable}.
\end{example}

\begin{proposition}\label{prop:extendable}
    Let $Q$ be a $p$-complete $\bE_\infty$-ring and let $\F^\star_\ev Q$ be the $p$-complete even filtration on
    $Q$. If $\gr^0Q$ is a discrete commutative ring and if $\gr^iQ$ is a perfect $\gr^0Q$-module for each $i\in\bZ$,
    then $$\F^\star_\syn\bWhat_Q((-)_\perf)\colon\Poly_{\bF_p}^\omega\rightarrow\CAlg(\FDhat(\F^\star_\ev
    Q)_p^\wedge)$$ is right extendable.
\end{proposition}

\begin{proof}
    The forgetful functor $\CAlg(\FDhat(\F^\star_\ev Q)_p^\wedge)\rightarrow\FDhat(\F^\star_\ev
    Q)_p^\wedge$ preserves
    limits and sifted colimits. So, we can test right extendability after forgetting the
    $\bE_\infty$-algebra structure by Lemma~\ref{lem:extendability}(a). Consider the commutative diagram
    $$\xymatrix{
        \Vect_{\bF_p}^\omega\ar[rrrd]_{\bW_{\gr^0Q}(\Sym_{\bF_p}^\perf(-))}\ar[rr]^{\F^\star_\ev\bW_Q(\Sym_{\bF_p}^\perf(-))}&&\FD(\F^\star_\ev
        Q)_p^\wedge\ar[r]^{\mathrm{complete}}&\FDhat(\F^\star_\ev
        Q)_p^\wedge\ar[rr]^{\gr^\star}&&\Gr\D(\gr^\star_\ev Q)_p^\wedge\ar[d]^{\mathrm{ev}^0(-)/p}\\
        &&&\D(\gr^0_\ev Q)_p^\wedge\ar[urr]^{(-)\otimes_{\gr^0_\ev Q}\gr^\star_\ev
        Q}\ar[rr]^{(-)\otimes_{\gr^0_\ev Q}\gr^0_\ev Q/p}&&\D(\gr^0_\ev Q/p),
    }$$
    where $\gr^0_\ev Q/p$ is the derived mod $p$ reduction and is in particular bounded above.
    The bottom composition from $\Vect_{\bF_p}^\omega$ to $\D(\gr^0_\ev Q/p)$ is right extendable by
    Lemma~\ref{lem:extendability}(b) as it factors through the extension of scalars functor $\D(\bF_p)\rightarrow\D(\gr^0_\ev Q/p)$.
    This also implies right extendability of the functor to $\D(\gr^0_\ev Q)_p^\wedge$ by
    completeness, since $\gr^0_\ev Q/p$ is perfect as a $\gr^0_\ev Q$-module
    and so tensoring with it commutes with all limits and colimits and is conservative.
    Right extendability of the composition to $\Gr\D(\gr^\star_\ev Q)_p^\wedge$ now follows from the
    assumption that each $\gr^i_\ev Q$ is perfect as a $\gr^0_\ev Q$-module. As $\gr^\star$ is
    conservative and commutes with all limits and colimits, it follows that the functor to
    $\FDhat(\F^\star_\ev Q)_p^\wedge$ is right extendable.
\end{proof}

\begin{theorem}\label{thm:nonconnective}
    Let $Q$ be a $p$-complete $\bE_\infty$-ring and let $\F^\star_\ev Q$ be the $p$-complete even filtration on
    $Q$. If $\gr^0Q$ is a discrete commutative ring and if $\gr^iQ$ is a perfect $\gr^0Q$-module for each $i\in\bZ$,
    then there is a nonconnective complete $p$-complete synthetic $Q$-based Witt vectors functor
    $\F^\star_\syn\bWhat_Q(-)\colon\DAlg_{\bF_p}^\perf\rightarrow\CAlg(\FDhat(\F^\star_\ev
    Q)_p^\wedge)$. Forgetting the filtration induces a nonconnective $Q$-based Witt vectors functor
    $\bWhat_Q\colon\DAlg_{\bF_p}^\perf\rightarrow\CAlg_Q^{\bE_\infty}$. Let
    $\bW=\bW_{\bZ_p}=\bWhat_{\bZ_p}$. The nonconnective $Q$-based Witt vectors satisfy the following properties:
    \begin{enumerate}
        \item[{\em (i)}] $\bW(-)\otimes_{\bZ_p}\bF_p$ factors canonically through
            $\DAlg_{\bF_p}^\perf$ and the induced functor
            $\DAlg_{\bF_p}^\perf\rightarrow\DAlg_{\bF_p}^\perf$ is the identity;
        \item[{\em (ii)}] $\F^\star_\syn\bWhat_Q(-)$ preserves sifted colimits and finite totalizations;
        \item[{\em (iii)}] the restriction of $\F^\star_\syn\bWhat_Q$ to $\CAlg_{\bF_p}^\perf$ agrees with the
            functor constructed via transmutation and flat filtrations in
            Definition~\ref{def:synqw};
        \item[{\em (iv)}] there are natural equivalences $(\bW(k)\otimes_{\bZ_p}\gr^\star_\ev
            Q)_p^\wedge\we\gr^\star_\syn\bWhat_Q(k)$
            for all $k\in\DAlg_{\bF_p}^\perf$.
    \end{enumerate}
\end{theorem}

\begin{proof}
    Let
    $F_Q\colon\DAlg_{\bF_p}\rightarrow\CAlg(\FDhat(\F^\star_\ev Q)_p^\wedge)$
    be the non-linear right left extension of the restriction
    of $\F^\star_\syn\bWhat_Q((-)_\perf)$ to $\Poly_{\bF_p}^\omega$ as
    in~\cite[Const.~2.2.17]{holeman-derived}, which applies because the restriction is right
    left extendable by Proposition~\ref{prop:extendable}.
    We will show that $F_Q$ factors through the localization
    $\DAlg_{\bF_p}\rightarrow\DAlg_{\bF_p}^\perf$ and we will write $\F^\star_\syn\bWhat_Q$ for the
    induced functor $\DAlg_{\bF_p}^\perf\rightarrow\CAlg(\FDhat(\F^\star_\ev Q)_p^\wedge)$.
    
    In the case of $Q=\bZ_p$ the filtration is trivial and we view $F=F_{\bZ_p}$ as taking
    values in $p$-complete $\bE_\infty$-$\bZ_p$-algebras. As derived commutative rings are closed
    under limits and colimits in $\bE_\infty$-$\bZ_p$-algebras and as
    $\DAlg_{\bF_p}^\perf\subseteq\DAlg_{\bF_p}$ is closed under limits and colimits, one obtains a
    canonical refinement $\DAlg_{\bF_p}\rightarrow\DAlg_{\bF_p}^\perf$ of $F\otimes_{\bZ_p}\bF_p$.
    Call the resulting functor $F_{\bF_p}$.
    As $(-)\otimes_{\bZ_p}\bF_p\colon\DAlg_{\bZ_p}^{\wedge}\rightarrow\DAlg_{\bF_p}$
    preserves colimits and limits, it follows that $F_{\bF_p}$ is the non-linear right left Kan extension of the
    functor $(-)_\perf\colon\Poly_{\bF_p}^\omega\rightarrow\DAlg_{\bF_p}^\perf$ to $\DAlg_{\bF_p}$.
    It must agree with the localization $(-)_\perf\colon\DAlg_{\bF_p}\rightarrow\DAlg_{\bF_p}^\perf$
    by Example~\ref{ex:perfection}. This proves that $F_{\bZ_p}$ factors through 
    the localization and proves (i).

    Arguing using associated graded pieces as in the proof of
    Proposition~\ref{prop:extendable} shows that $\gr^\star
    F_Q(k)\we(\bW(k)\otimes_{\bZ_p}\gr^\star_\ev Q)_p^\wedge$ for $k\in\DAlg_{\bF_p}$. Thus, $F_Q$
    factors through $(-)_\perf$, as desired. This shows that the functor $\F^\star_\syn\bWhat_Q$
    exists and also establishes (iv).

    The functor $\F^\star_\syn\bWhat_Q(-)$ preserves
    finite totalizations by arguing that $k\mapsto \gr^0_\ev\bWhat_Q(k)/p\we k\otimes_{\bF_p}\gr^0_\ev Q/p$
    does using~\cite[Thm.~2.2.20]{holeman-derived} and lifting to the entire associated graded as in
    the proof of Proposition~\ref{prop:extendable}. It preserves sifted colimits
    by construction~\cite[Const.~2.2.17]{holeman-derived}.
    When restricted to perfect $\bF_p$-algebras, the nonconnective spherical Witt vectors functor
    agrees with the one constructed in Definition~\ref{def:synqw} by~\cite[Obs.~2.2.19]{holeman-derived}. 
    This proves (ii) and (iii).
\end{proof}

\begin{example}
    We let $\F^\star_\syn\widehat{\bS\bW}=\F^\star_\syn\bWhat_{\bS_p}$.
\end{example}

\begin{remark}
    A nonconnective Witt vector functor $\bW$ has been constructed
    in~\cite[Def.~2.3.13]{holeman-derived}, but with target the $\infty$-category of derived
    $\delta$-rings. Forgetting the $\delta$-ring structure and restricted to perfect derived
    commutative $\bF_p$-algebras, one recovers our functor $\bW=\bW_{\bZ_p}=\bWhat_{\bZ_p}$, as we
    will see in Corollary~\ref{cor:compare}.
\end{remark}

\begin{corollary}\label{cor:swcochains}
    Let $Q$ be a $p$-complete $\bE_\infty$-ring and let $\F^\star_\ev Q$ be the $p$-complete even filtration on
    $Q$. Suppose that $\gr^0Q$ is a discrete commutative ring and that $\gr^iQ$ is a perfect $\gr^0Q$-module for each $i\in\bZ$.
    If $\F^\star_\ev Q$ is complete and exhaustive and
    if $X\in\Sscr^\omega$, then there is a natural equivalence
    $$\bWhat_Q(\bF_p^X)\we Q^X.$$
\end{corollary}

\begin{proof}
    This follows from the fact that $\bF_p^X$ is a finite totalization (as $X$ can be realized by a finite geometric
    realization), the fact that $\F^\star_\syn\bWhat_Q$ commutes with finite totalizations by
    Theorem~\ref{thm:nonconnective}(i), and the
    fact that the colimit functor $\FDhat(\F^\star_\ev Q)\rightarrow\D(Q)$ preserves finite limits.
\end{proof}

We return to this theme in Section~\ref{sec:synthetic}.

\section{Rational spaces}\label{sec:rational}

We recall a part of Sullivan's cochain-theoretic approach~\cite{sullivan-infinitesimal} to rational homotopy theory,
ignoring the crucial role of minimal models. For more details,
see~\cite{bousfield-gugenheim,dag13}.

\begin{definition}[$\bQ$-finite spaces]
    Let $\Sscr_{\bQ\fin}$ be the
    $\infty$-category of $\bQ$-finite spaces: a space is $\bQ$-finite if it is $n$-truncated for some
    $n$, it has finitely many components, and each component $Y\subseteq X$ satisfies the following
    conditions:
    \begin{enumerate}
        \item[(1)] $\pi_1(Y)$ is nilpotent and admits a finite central series with finite-dimensional $\bQ$-vector
            spaces as graded pieces;
        \item[(2)] $\pi_1(Y)$ acts nilpotently on each $\pi_i(Y)$ for $i\geq 2$;
        \item[(3)] each $\pi_i(Y)$ is a
            finite-dimensional $\bQ$-vector space for $i\geq 2$.
    \end{enumerate}
\end{definition}

\begin{definition}[Stone algebras]
    Let $k$ be a commutative ring whose only idempotents are $0$ and $1$. We let $\CAlg_k^\Stone\subseteq\CAlg_k$ be the smallest 
    subcategory of $\CAlg_k$ generated under filtered colimits by the commutative $k$-algebras $k^S$ where
    $S$ is finite. The objects of $\CAlg_k^\Stone$ are called Stone $k$-algebras.
    As the class of algebras $k^S$ consists of compact objects and is closed under
    finite colimits and retracts
    in $\CAlg_k$, it follows that $\CAlg_k^\Stone$ is presentable and compactly generated by the
    $k^S$ where $S$ is finite. The inclusion $\CAlg_k^\Stone\rightarrow\CAlg_k$ admits a right
    adjoint. Let $\DAlg_k^\Stone\subseteq\DAlg_k^\ccn$ be the full subcategory of
    coconnective derived commutative $k$-algebras $R$ such that $\pi_0R$ is a Stone $k$-algebra.
\end{definition}

\begin{warning}
    While $\CAlg_{\bQ}^\Stone$ is presentable, $\DAlg_{\bQ}^\Stone$ is not. For example, the pushout
    of $\bQ\leftarrow\bQ^{S^1}\rightarrow\bQ$ does not exist in $\DAlg_{\bQ}^\Stone$. Thus,
    $\DAlg_\bQ^\Stone\subseteq\DAlg_\bQ$ is an accessible subcategory, but the inclusion does not
    admit a right adjoint.
\end{warning}

\begin{remark}
    If $k$ is a commutative $\bQ$-algebra, then $\DAlg_k^{\ccn}$ can be understood via the theory of
    commutative differential graded $k$-algebras.
\end{remark}

The category of Stone $k$-algebras does not depend on $k$.
Let $\Fin$ denote the category of finite sets.

\begin{lemma}\label{lem:stone}
    For any commutative ring whose only idempotents are $0$ and $1$,
    there is a natural equivalence $\CAlg_k^\Stone\we\Pro(\Fin)^\op$.
\end{lemma}

\begin{proof}
    Indeed, the category of commutative $k$-algebras of the form $k^S$ for $S$ a finite set is
    equivalent to $\Fin^\op$ and $\Ind(\Fin^\op)\we\Pro(\Fin)^\op$.
    One reduces to checking that $\Hom_k(k^{\{0,1\}},k)\iso\{0,1\}$.
\end{proof}

\begin{theorem}[Sullivan]\label{thm:sullivan}
    The functor $\Sscr^\op_{\bQ\fin}\rightarrow\DAlg_{\bQ}^\ccn$ given by $X\mapsto\bQ^X$ factors
    through $\DAlg_\bQ^\Stone$ and extends
    to a fully faithful functor $\Pro(\Sscr_{\bQ\fin})^\op\rightarrow\DAlg_{\bQ}^\Stone$.
\end{theorem}

\begin{proof}
    If $X\in\Sscr_{\bQ\fin}$, then $X$ has finitely many connected components, so
    $\H_0(\bQ^X)\iso\H^0(X,\bQ)$ is a Stone $\bQ$-algebra. This proves that the functor lands in
    $\DAlg_{\bQ}^\Stone$. Consider now the left Kan extension
    $\Pro(\Sscr_{\bQ\fin})^\op\rightarrow\DAlg_{\bQ}^\Stone$.
    For $X,Y\in\Sscr$,
    both sides of $\Map_\Sscr(Y,X)\rightarrow\Map_{\DAlg_{\bQ}^\ccn}(\bQ^X,\bQ^Y)$
    take colimits of $Y$ to limits, so one can reduce fully faithfulness on $\Sscr_{\bQ\fin}^\op$ to
    the case where $Y$ is contractible. In other words, for $X\in\Sscr_{\bQ\fin}$, we want to show
    that the map $X\rightarrow\Map_{\DAlg_{\bQ}^\ccn}(\bQ^X,\bQ)$ is an equivalence.
    If $X$ is an Eilenberg--Mac Lane space of type $K(V,n)$ for $n\geq 1$ for $V$
    finite-dimensional, the theorem holds by reduction to the case of $V\we\bQ$ by K\"unneth and
    then to $n=1$ by Eilenberg--Moore.
    By nilpotence and induction, we can assume that $X$ is obtained as a
    pullback
    $$\xymatrix{
        X\ar[r]\ar[d]&\ast\ar[d]\\
        Z\ar[r]&K(V,n+1),
    }$$
    where $V$ is a finite-dimensional $\bQ$-vector space, the natural map
    $Z\rightarrow\Map_{\DAlg_{\bQ}^\ccn}(\bQ^Z,\bQ)$ is an equivalence, and $n\geq 1$.
    Nilpotence and the Eilenberg--Moore theorem (see for example~\cite[Cor.~1.1.10]{dag13}) guarantee that
    $\bQ\otimes_{\bQ^{K(V,n+1)}}\bQ^Z\rightarrow \bQ^X$ is an equivalence, which gives the
    inductive step. This proves fully faithfulness on $\Sscr_{\bQ\fin}^\op$. For each $n\geq 1$ and
    each finite set $S$,
    $\bQ^{K(\bQ,n)}$ and $\bQ^S$ are compact in $\DAlg_{\bQ}^\ccn$, from which it follows that the functor is fully faithful on all of
    $\Pro(\Sscr_{\bQ\fin})^\op$.
\end{proof}

\begin{remark}
    We expect that $\Pro(\Sscr_{\bQ\fin})\rightarrow\DAlg_{\bQ}^\Stone$ is an equivalence, but do
    not have a proof of this at the moment.
\end{remark}

Sullivan's theorem gives a template we will follow to study $p$-adic and integral models for spaces.
It is more straightforward to establish because the rational cohomology of the Eilenberg--Mac Lane spaces
$K(\bQ,n)$ for $n\geq 1$ is easy to understand by iterating the Leray--Serre spectral sequence
starting from $n=1$ and because we know by the theory of rational cdgas
that $\bQ^{K(\bQ,n)}$ is compact.

\begin{definition}[Continuous rational cohomology]
    Given a pro-$\bQ$-finite space $X$, we write $\bQ^X$ for the image of $X$ under the equivalence
    of Theorem~\ref{thm:sullivan}. This derived commutative $\bQ$-algebra computes the continuous
    rational cohomology of $X$. Specifically, if we write $X\we\{\lim_i X_i\}$ as a cofiltered diagram where each $X_i$ is
    $\bQ$-finite, then $\bQ^X\we\colim_i\bQ^{X_i}$. If $X\we\lim\tau_{\leq i}X$ where each
    $\tau_{\leq i}X$ is $\bQ$-finite, then the continuous rational cohomology of the tower
    $\{\tau_{\leq i}X\}$ agrees with
    the usual rational cohomology of the limit.
\end{definition}

\section{$p$-adic spaces}\label{sec:padic}

In this section, we fix a prime $p>0$ and we give several $p$-adic analogs of Sullivan's theorem.

\subsection{Derived $p$-Boolean rings}\label{sub:pboolean}

Let $k$ be a perfect ring of characteristic $p$. The full subcategory
$\CAlg_k^\perf\subseteq\CAlg_k$ of perfect commutative $k$-algebras admits a derived analog.

\begin{definition}[Perfect derived commutative $k$-algebras]
    Let $k$ be a perfect ring of characteristic $p$ and let $\DAlg_k$ be the $\infty$-category of
    derived commutative $k$-algebras.
    In~\cite[Const~2.4.1]{holeman-derived}, Holeman shows that every derived commutative $\bF_p$-algebra admits a Frobenius
    $\varphi$ endomorphism, giving an action of $\B\bN$ on $\DAlg_{\bF_p}$ (since the Frobenius on
    $\bF_p$ is the identity).
    This fact also follows
    from~\cite[Thm.~5.20]{brantner-campos-nuiten}. A derived commutative $k$-algebra is
    perfect if its Frobenius is an equivalence. Let $\DAlg_k^\perf\subseteq\DAlg_k$ be the
    full subcategory of perfect derived commutative $k$-algebras.
\end{definition}

\begin{lemma}
    Let $k$ be a perfect commutative $\bF_p$-algebra.
    Any perfect derived commutative $k$-algebra is coconnective.
    The inclusion $$\CAlg_{k}^\perf\subseteq\DAlg_{k}^\perf$$ admits a right adjoint given by $\pi_0$.
\end{lemma}

\begin{proof}
    Indeed, if $R\in\DAlg_{k}^\perf$, then the natural map $\tau_{\geq 0}R\rightarrow\pi_0R$ is
    an equivalence by Lemma~\ref{lem:animated} as $\tau_{\geq 0}R$ is a perfect animated commutative
    $k$-algebra.
\end{proof}

\begin{lemma}\label{lem:adjoints1}
    Let $k$ be a perfect commutative $\bF_p$-algebra.
    The forgetful functors $\CAlg_{k}^\perf\rightarrow\CAlg_{k}$ and
    $\DAlg_{k}^\perf\rightarrow\DAlg_{k}$ admit both left and right adjoints.
\end{lemma}

\begin{proof}
    This follows from the fact that all limits and colimits of perfect derived commutative
    $k$-algebras are, when computed in $\DAlg_{k}$, still perfect. The left adjoint is given
    by the colimit perfection $R\mapsto R_\perf=\colim_\varphi R$ and the right adjoint is given by the inverse limit
    perfection $R\mapsto R^\perf=\lim_\varphi R$.
\end{proof}

\begin{corollary}\label{cor:monadic1}
    The $\infty$-categories $\CAlg_{k}^\perf$ and $\DAlg_{k}^\perf$ are monadic over
    $\Mod_{k}$ and $\D(k)$, respectively.
\end{corollary}

\begin{proof}
    We can consider the forgetful functor
    $$\DAlg_{k}^\perf\rightarrow\DAlg_{k}\rightarrow\D(k).$$
    The composition preserves all limits and sifted colimits. The result follows from Barr--Beck
    (see~\cite[Thm.~4.7.3.5]{ha}).
\end{proof}

\begin{remark}
    The actions of $\B\bN$ on $\CAlg_{\bF_p}$ and $\DAlg_{\bF_p}$ extend to actions of $\B\bZ$
    on $\CAlg_{\bF_p}^\perf$ and $\DAlg_{\bF_p}^\perf$.
\end{remark}

\begin{definition}[$p$-Boolean rings]
    Let $\CAlg_{\bF_p}^{\varphi=1}\subseteq\CAlg_{\bF_p}^\perf\subseteq\CAlg_{\bF_p}$ be the full
    subcategory of commutative $\bF_p$-algebras $R$ where $\varphi(x)=x^p=x$ for all $x\in R$. We call
    the objects $p$-Boolean rings since when $p=2$ these coincide with the Boolean rings in the
    usual sense.
\end{definition}

\begin{lemma}\label{lem:fixed}
    The natural inclusion
    $$\CAlg_{\bF_p}^{\varphi=1}\rightarrow(\CAlg_{\bF_p}^\perf)^{\h S^1}$$ is an equivalence of
    categories.
\end{lemma}

This motivates the following definition.

\begin{definition}[$p$-Boolean derived rings]
    A $p$-Boolean derived ring is an object of $\DAlg_{\bF_p}^{\varphi=1}$, which we define as
    $(\DAlg_{\bF_p}^\perf)^{\h S^1}$.
\end{definition}

\begin{lemma}\label{lem:adjoints2}
    The forgetful functors $\CAlg_{\bF_p}^{\varphi=1}\rightarrow\CAlg_{\bF_p}^\perf$ and
    $\DAlg_{\bF_p}^{\varphi=1}\rightarrow\DAlg_{\bF_p}^\perf$ admit both left and right adjoints.
\end{lemma}

\begin{proof}
    As $\Pr^\L$ is closed under limits, we can view the forgetful functors as morphisms in $\Pr^\L$.
    Hence, they admit right adjoints. Using the reasoning of the proof
    of~\cite[Lem.~7.3]{yuan-integral}, especially the appeal to~\cite[Cor.~2.11]{quigley-shah}, these right adjoints are computed by $$R\mapsto
    \ker(R\xrightarrow{1-\varphi}R)\quad\text{and}\quad R\mapsto\fib(R\xrightarrow{1-\varphi}R),$$
    respectively. Now, consider the functor
    $\DAlg_{\bF_p}^\perf\rightarrow(\DAlg_{\bF_p}^\perf)_{\h S^1}$ in $\Pr^\L$ where the homotopy
    $S^1$ orbits are computed in $\Pr^\L$. By~\cite[Thm.~5.5.3.18]{htt}, the recipe for computing
    this colimit is by passing to to right adjoints and computing the limit in $\Pr^\R$ or,
    equivalently, in $\widehat{\Cat}_\infty$. But, under the equivalence $\B\bZ\we(\B\bZ)^\op$
    of $\bE_1$-spaces,
    this limit is equivalent to $(\DAlg_{\bF_p}^\perf)^{\h S^1}$ and under this identification the forgetful functor is right
    adjoint to the projection to the homotopy fixed points. It follows that the map
    $\DAlg_{\bF_p}^{\varphi=1}\rightarrow\DAlg_{\bF_p}^\perf$ is also a right adjoint.
    The case of discrete commutative $\bF_p$-algebras is similar.
    The left adjoints are computed by taking $\bZ$-orbits in the relevant categories:
    $$R\mapsto R/(x-x^p:x\in R)\quad\text{and}\quad R\mapsto\mathrm{coeq}(\id,\varphi\colon
    R\rightrightarrows R).$$
    Here, the coequalizer is computed in $\DAlg_{\bF_p}^\perf$.
\end{proof}

\begin{corollary}\label{cor:monadic2}
    The $\infty$-categories $\CAlg_{\bF_p}^{\varphi=1}$ and $\DAlg_{\bF_p}^{\varphi=1}$ are monadic
    over $\Mod_{\bF_p}$ and $\D(\bF_p)$, respectively.
\end{corollary}

\begin{proof}
    Follow the proof of Corollary~\ref{cor:monadic1} using Lemma~\ref{lem:adjoints2}.
\end{proof}

\begin{example}\label{ex:free}
    The free $p$-Boolean derived ring on $n$ elements is discrete and is equivalent to
    $\bF_p^{\times p^n}$. It is enough to prove this for $n=1$ by counting
    dimensions. When $n=1$, the free $p$-Boolean ring is $\bF_p[x]/(x^p-x)$. The
    ideal in question is generated by $\prod_{a\in\bF_p}(x-a)$ so the quotient ring is isomorphic to
    $\prod_{a\in\bF_p}\bF_p$, as desired. Kubrak has pointed out that, more canonically, the free $p$-Boolean ring on a
    set $S$ is isomorphic to the group algebra of the abelian group obtained as an $S$-indexed coproduct of
    $\bF_p$.
\end{example}

\begin{remark}[Categories of models I]
    Let $\c\CAlg_{\bF_p}^\perf$ and $\c\CAlg_{\bF_p}^{\varphi=1}$ be the categories of cosimplicial
    perfect $\bF_p$-algebras and cosimplicial $p$-Boolean rings, respectively. Each admits the
    structure of a model category where the weak equivalences are obtained by the Dold--Kan
    correspondence. We expect that the resulting $\infty$-categories $\c\CAlg_{\bF_p}^\perf[W^{-1}]$ and
    $\c\CAlg_{\bF_p}^{\varphi=1}[W^{-1}]$ are equivalent to $\DAlg_{\bF_p}^\perf$ and
    $\DAlg_{\bF_p}^{\varphi=1}$. A proof could likely use the forthcoming work of Mathew and
    Mondal~\cite{mathew-mondal}, or~\cite[Thm.~5.20]{brantner-campos-nuiten}.
\end{remark}

The following result can be found in~\cite{stringall}; see also~\cite[Prop.~A.1.12]{dag12}.

\begin{theorem}[Stringall]\label{thm:comparison}
    For any primes $\ell,p$, there is a equivalence
    $\CAlg_{\bF_\ell}^{\varphi=1}\we\CAlg_{\bF_p}^{\varphi=1}$, i.e., the categories of $p$-Boolean
    rings are equivalent for all $p$.
\end{theorem}

\begin{proof}
    In fact, the category of $p$-Boolean rings agrees with the category of Stone $\bF_p$-algebras,
    so this is a special case of Lemma~\ref{lem:stone}.
\end{proof}

\begin{remark}
    Example~\ref{ex:free} shows that the categories of $p$-Boolean rings are not equivalent as
    concrete categories, which is it say that the induced monads on the category of sets are not
    equivalent. On the other hand, there is a nice way of expressing the equivalence in the theorem
    using Stone duality. Recall that there is an equivalence
    $$\Spec\colon\CAlg_{\bF_2}^{\varphi=1}\we\mathrm{Stone}^\op\rcolon C(-,\bF_2),$$
    where $\mathrm{Stone}$ is the category of Stone spaces (i.e., profinite spaces) and continuous
    maps and $C(-,\bF_2)$ denotes the functor of continuous maps to $\bF_2$. The equivalence above
    is explained by an equivalence
    $$\Spec\colon\CAlg_{\bF_p}^{\varphi=1}\we\mathrm{Stone}^\op\rcolon C(-,\bF_p)$$
    for any prime $p$.\footnote{Note that the compact generation of $p$-Boolean rings by finite
    dimensional $\bF_p$-algebras implies that $\Spec R$ is profinite as for a $p$-Boolean ring
    $\Spec R$ is given by $\Hom_{\CAlg_{\bF_p}^{\varphi=1}}(-,\bF_p)$.} A specific incarnation of the equivalence
    $$\CAlg_{\bF_\ell}^{\varphi=1}\rightarrow\CAlg_{\bF_p}^{\varphi=1}$$ is thus given by $R\mapsto
    C(\Spec R,\bF_p)$ for an $\ell$-Boolean $\bF_\ell$-algebra $R$.
\end{remark}

As in Yuan and Mandell, we can unwind the Frobenius.

\begin{lemma}\label{lem:unwind}
    The composite functors
    $$\CAlg_{\bF_p}^{\varphi=1}\rightarrow\CAlg_{\bF_p}^\perf\xrightarrow{(-)\otimes_{\bF_p}\overline{\bF}_p}\CAlg_{\overline{\bF}_p}^\perf$$
    and
    $$\DAlg_{\bF_p}^{\varphi=1}\rightarrow\DAlg_{\bF_p}^\perf\xrightarrow{(-)\otimes_{\bF_p}\overline{\bF}_p}\DAlg_{\overline{\bF}_p}^\perf$$
    are fully faithful.
\end{lemma}

\begin{proof}
    The unit of the adjunction is given by
    $R\rightarrow(R\otimes_{\bF_p}\overline{\bF}_p)^{\varphi=1}$, which is an equivalence.
\end{proof}

\begin{remark}
    By taking spherical Witt vectors of $p$-Boolean $\bF_p$-algebras and tensoring up to
    $\bS_p^\un$, we find categories of $p$-complete $\bE_\infty$-algebras over $\bS_p^\un$ which do
    not depend on $p$.
\end{remark}

\subsection{The $p$-adic cochain theorems}\label{sub:padic_cochain}

Fix a prime number $p$. Let $\Sscr_{p\nil}\subseteq\Sscr$ be the full subcategory of $p$-complete
nilpotent spaces. Let $\Sscr_{p\ft}\subseteq\Sscr_{p\nil}$ be the full subcategory of $p$-complete
nilpotent spaces of finite type. Let $\Sscr_{p\fin}\subseteq\Sscr_{p\ft}$ be the full subcategory of
truncated $p$-finite type spaces, which we will call $p$-finite spaces. Let
$\Sscr_{p}^\omega\subseteq\Sscr_{p\ft}$ be the full subcategory of compact $p$-nilpotent
spaces. These are the $p$-nilpotent spaces with finite $\bF_p$-homology. Finally, let $\Sscr_{p,\geq
2}^\omega\subseteq\Sscr_p^\omega$ be the full subcategory of simply connected compact $p$-nilpotent
spaces.

Here is our version of the theorems of Mandell~\cite{mandell-padic} and Kriz~\cite{kriz}.

\begin{theorem}[The mod $p$ cochain theorem]\label{thm:derivedstone}
    The functor $\Sscr_{p\fin}^\op\rightarrow\DAlg_{\bF_p}^{\varphi=1}$ given by taking $X$ to
    $\bF_p^X$ is fully faithful and extends to an equivalence
    $\Pro(\Sscr_{p\fin})^\op\we\Ind(\Sscr_{p\fin}^\op)\we\DAlg_{\bF_p}^{\varphi=1}$.
\end{theorem}

We need a preliminary proposition.

\begin{proposition}\label{prop:free_derived_boolean}
    For $n\geq 1$, the $p$-Boolean derived ring $\bF_p^{K(\bF_p,n)}$ is equivalent to the free
    $p$-Boolean derived ring on $\bF_p[-n]$; in particular, it is a compact object of
    $\DAlg_{\bF_p}^{\varphi=1}$.
\end{proposition}

\begin{proof}
    To\"en introduced the notion of affine stacks in~\cite{toen-affines}. A higher stack
    $X\colon\CAlg_{\bF_p}\rightarrow\Sscr$ is an affine stack if it is represented by a cosimplicial
    commutative ring $S^\bullet$, so that $X(R)\we|\Hom_{\CAlg_{\bF_p}}(S^\bullet,R)|$ for
    $R\in\CAlg_{\bF_p}$. To\"en proved in~\cite{toen-affines} that induced functor $\c\CAlg_{\bF_p}^\op\rightarrow\mathrm{Stk}_{\bF_p}$
    factors through the localization of cosimplicial commutative $\bF_p$-algebras at the weak
    homotopy equivalences and induces a fully faithful functor
    $\c\CAlg_{\bF_p}^\op[W^{-1}]\rightarrow\mathrm{Stk}_{\bF_p}$. (This holds more generally over
    an arbitrary commutative ring.)

    Forthcoming work~\cite{mathew-mondal} of Mathew and Mondal proves that
    $\c\CAlg_{\bF_p}^\op[W^{-1}]\we\DAlg_{\bF_p}^{\mathrm{ccn}}$, the full subcategory of
    coconnective derived commutative $\bF_p$-algebras. (Brantner, Campos, and Nuiten also prove
    in~\cite{brantner-campos-nuiten} that $\DAlg_{\bF_p}$ is equivalent to a model category
    structure on simplicial cosimplicial $\bF_p$-algebras.) The affine stacks are of the form $\Spec
    B=\Map_{\DAlg_{\bF_p}}(B,-)$ where $B\in\DAlg_{\bF_p}^{\mathrm{ccn}}$.

    To\"en also identifies the essential image of the affine stack functor. It contains precisely
    those higher stacks all of whose homotopy sheaves $\pi_iX$ are unipotent affine group schemes
    over $\bF_p$ for $i>0$. A particular example is the constant sheaf of spaces associated to
    the Eilenberg--Mac Lane space $K(\bF_p,n)$.

    Now, we compute the affine stack associated to $K(\bF_p,n)$ in two different ways. First, we
    have left adjoint functors
    $$\Sscr\rightarrow\Stk_{\bF_p}\xrightarrow{\R\Gamma(-,\Oscr)}\DAlg_{\bF_p}^{\ccn,\op},$$
    where the first functor takes a space $X$ to the sheaf of spaces $\underline{X}$ associated to
    the constant presheaf on $X$. The right adjoint to $\R\Gamma(-,\Oscr)$ is $\Spec$. Taking
    $\R\Gamma(-,\Oscr)$ is a localization by To\"en's theorem with unit map
    $Y\rightarrow\Spec\R\Gamma(Y,\Oscr)$. In particular, if $X$ is a space such that $\underline{X}$
    is an affine stack, then the unit map
    $\underline{X}\rightarrow\Spec\R\Gamma(\underline{X},\Oscr)$ is an equivalence. Now, suppose
    that $X\we|X_\bullet|$ where $X_\bullet$ is a simplicial space. Then, since the maps above
    preserve colimits, we have
    $\R\Gamma(\underline{X},\Oscr)\we\Tot\R\Gamma(\underline{X_\bullet},\Oscr)$. If each $X_n$ is
    finite for $n\geq 0$, then for each $n$ we have
    $\R\Gamma(\underline{X_n},\Oscr)\we\R\Gamma(*,\Oscr)^{X_n}\we\bF_p^{X_n}$, so we can rewrite the $\Tot$ as $\Tot\bF_p^{X_\bullet}$, which computes
    the function spectrum $\bF_p^X$.

    In particular, we know that $\underline{K(\bF_p,n)}$ is an affine stack by the unipotence
    criterion and it follows from the above
    argument that it is equivalent to $\Spec\bF_p^{K(\bF_p,n)}$. On the other hand, we have a fiber
    sequence of sheaves $$\underline{K(\bF_p,n)}\rightarrow
    K(\Oscr,n)\xrightarrow{\id-\varphi}K(\Oscr,n)$$
    for each $n\geq 0$ arising from the Artin--Schreier sequence. Let $F_n=\R\Gamma(K(\Oscr,n),\Oscr)$, which is the free derived commutative ring on
    a degree $-n$ class $x_{-n}$. By an Eilenberg--Moore argument (see~\cite[Cor.~2.3.5]{mondal-reinecke}), it
    follows that for $n\geq 1$ the natural map
    $$F_n\otimes_{F_n}\bF_p\rightarrow\R\Gamma(\underline{K(\bF_p,n)},\Oscr)$$
    is an equivalence, where the map $F_n\rightarrow F_n$ takes $x_{-n}$ to $x_{-n}-\varphi(x_{-n})$.
    By the universal property of $F_n$, we see that $$\bF_p^{K(\bF_p,n)}\we (F_n)_{\varphi=1}.$$
    Now, for any derived commutative ring $R$, there is a natural equivalence
    $$(R_\perf)_{\varphi=1}\we(R_{\varphi=1})_\perf\we R_{\varphi=1},$$
    which is to say that if we force $\varphi$ to be the identity, then the result is certainly
    perfect. It follows that $\bF_p^{K(\bF_p,n)}$ has the desired universal property as a free
    object. Compactness follows from the fact that
    $\DAlg_{\bF_p}\leftarrow\DAlg_{\bF_p}^{\varphi=1}$ preserves all colimits and in particular
    filtered colimits.
\end{proof}

\begin{proof}[Proof of Theorem~\ref{thm:derivedstone}]
    As the $\infty$-category of $p$-finite spaces has all finite limits and is idempotent complete,
    $\Sscr_{p\fin}^\op$ has all finite colimits and is idempotent complete;
    $\Ind(\Sscr_{p\fin}^\op)$ is presentable. The objects $\bF_p^{K(\bF_p,n)}$ together with
    the $p$-Boolean rings $\bF_p^S$ for $S$ finite give a set of compact generators of
    $\DAlg_{\bF_p}^{\varphi=1}$ using Proposition~\ref{prop:free_derived_boolean}, which also
    implies that these
    objects are in the image of the functor in question. To prove the theorem, it is enough to
    see fully faithfulness on $\Sscr_{p\fin}^\op$. Thus, fix $X\in\Sscr_{p\fin}$, $Y\in\Sscr$, and consider
    the natural map $\Map_{\Sscr}(Y,X)\rightarrow\Map_{\DAlg_{\bF_p}^\varphi=1}(\bF_p^X,\bF_p^Y)$.
    Both sides take colimits in $Y$ to limits, so we can reduce to the case where $Y$ is a point.
    Assume that $X$ is connected. Since it is also $p$-finite, and hence nilpotent, it is built out
    of finitely many iterated pullbacks the form
    $$\xymatrix{
        A\ar[r]\ar[d]&\ast\ar[d]\\
        B\ar[r]&K(\bF_p,n+1)
    }$$
    for $n\geq 1$. As $\bF_p^{(-)}$ takes such pullback squares to pushout squares, we reduce (when
    $X$ is connected) to Eilenberg--Mac Lane spaces, which follows from
    Proposition~\ref{prop:free_derived_boolean}. Indeed, by the universal property of
    $\bF_p^{K(\bF_p,n)}$ for $n\geq 1$ established in the proposition,
    $$\Map_{\DAlg_{\bF_p}^{\varphi=1}}(\bF_p^{K(\bF_p,n)},\bF_p)\we\Map_{\D(k)}(\bF_p[-n],\bF_p)\we
    K(\bF_p,n).$$
    It remains to handle the reduction to the connected case. For this, note that if $X$ and $Y$ are
    connected, then any map $\bF_p^{X\coprod Y}\we\bF_p^X\times\bF_p^Y\rightarrow\bF_p$ sends exactly one of the
    $X$ or $Y$ components to zero as it induces a map of rings
    $\pi_0(\bF_p^X\times\bF_p^Y)\iso\bF_p\times\bF_p\rightarrow\bF_p$. It follows
    that $\Map_{\DAlg_{\bF_p}^{\varphi=1}}(\bF_p^{X\coprod Y},\bF_p)\we X\coprod Y$ if $X$ and
    $Y$ are connected and $p$-finite. This completes the proof (by induction on the finite number of
    connected components).
\end{proof}

\begin{example}
    If $X$ is a scheme, then $\R\Gamma_\et(X,\bF_p)$ is naturally a $p$-Boolean derived ring and
    hence corresponds to a pro-$p$-finite space under the equivalence in
    Theorem~\ref{thm:derivedstone}, the $p$-adic \'etale homotopy type of $X$ in the sense of
    Artin--Mazur~\cite{artin-mazur} (see also~\cite[Sec.~4.4]{mondal-reinecke}).
\end{example}

\begin{corollary}
    The functor $X\mapsto\bFbar_p^X$ of continuous $\bFbar_p$-cochains
    $$\Pro(\Sscr_{p\fin})^\op\rightarrow\DAlg_{\bFbar_p}^\perf$$ is fully faithful.
\end{corollary}

\begin{proof}
    This follows from Theorem~\ref{thm:derivedstone} and Lemma~\ref{lem:unwind}.
\end{proof}

To go further, lifting the previous theorem from $\bF_p$ to $\bZ_p$, we use derived
$\delta$-rings. These were introduced by Holeman~\cite{holeman-derived}, who shows that one can derive the
free $\delta$-ring monad. Let $\DAlg_{\bZ_p}^{\delta,\wedge}$ be the $\infty$-category of $p$-complete
derived $\delta$-rings. As in the case of derived commutative rings there is a full subcategory
$\DAlg_{\bZ_p}^{\delta,\perf,\wedge}$ of $p$-complete derived $\delta$-rings for which the Frobenius
is an equivalence, and we set
$\DAlg_{\bZ_p}^{\delta,\varphi=1,\wedge}=(\DAlg_{\bZ_p}^{\delta,\perf,\wedge})^{\h S^1}$.

The forgetful functor $\DAlg_{\bZ_p}^{\delta,\wedge}\rightarrow\DAlg_{\bZ_p}^\wedge$ admits a right
adjoint, denoted by $\bW$, which is a derived version of the $p$-typical Witt vectors functor.
In Corollary~\ref{cor:compare} below we show that $\bW$
agrees (after forgetting the $\delta$-ring structure) with the functor defined via right left Kan
extension in Theorem~\ref{thm:nonconnective}.

\begin{lemma}\label{lem:discrete}
    If $R$ is a derived $p$-complete commutative ring, then $\bW(R)$ is discrete and hence equivalent
    to the usual ring of $p$-typical Witt vectors of $R$.
\end{lemma}

\begin{proof}
    It is enough to show that 
    $$\pi_i\Map_{\DAlg_{\bZ_p}^{\delta,\wedge}}(\LSym_{\bZ_p}^{\delta}(\bZ_p[-n])_p^\wedge,\bW(R))=0$$ 
    for all $n\geq 1$ and all $i<n$ or, by adjunction, that
    $$\pi_i\Map_{\DAlg_{\bZ_p}^{\wedge}}(\LSym_{\bZ_p}^{\delta}(\bZ_p[-n])_p^\wedge,R)=0$$ 
    for all $n\geq 1$ and all $i<n$. Let $F_n=\LSym_{\bZ_p}^\delta(\bZ_p[-n])_p^\wedge$.
    Let $X^\bullet$ be the result of the Dold--Kan correspondence applied to $\bZ_p[-n]$, so it is a
    cosimplicial $\bZ_p$-module which is finitely presented and free in each degree. Moreover, $X^i=0$
    for $0\leq i<n$. By construction, $F_n$ is equivalent to the totalization of the cosimplicial
    commutative ring $F_n^\bullet=\Tot\Sym^\delta(X^\bullet)_p^\wedge$. For $0\leq i<n$,
    $F_n^i=\bZ_p$. To compute $\Map_{\DAlg_{\bZ_p}^\wedge}(\LSym^\delta_{\bZ_p}(\bZ_p[-n])_p^\wedge,R)$, we can take a 
    cofibrant resolution $QF_n^\bullet$ of $F_n^\bullet$ and compute the geometric realization
    $|\Map_{\CAlg_{\bZ_p}^\wedge}(QF_n^\bullet,R)|$ by~\cite{mathew-mondal} and~\cite{toen-affines}.
    Such a cofibrant resolution $QF_n^\bullet$ can also be chosen to satisfy $QF_n^i=\bZ_p$ for
    $0\leq i<n$ using that for the class $S(n)\rightarrow D(n)$, $n\geq 0$, of
    generating cofibrations studied in~\cite[Thm.~2.1.2]{toen-affines} one has no non-trivial maps from $D(i)$
    to $F_n^\bullet$ for $0\leq i\leq n$. It follows that the geometric realization is $n$-connective, as desired.
\end{proof}

\begin{proposition}\label{prop:fully_faithful}
    \begin{enumerate}
        \item[{\em (a)}] There are equivalences
            $\DAlg_{k}^{\perf}\we\DAlg_{\bW(k)}^{\delta,\perf,\wedge}$
            for any perfect commutative $\bF_p$-algebra $k$
            and
            $\DAlg_{\bF_p}^{\varphi=1}\we\DAlg_{\bZ_p}^{\delta,\varphi=1,\wedge}$.
        \item[{\em (b)}] The functor
            $\DAlg_{\bZ_p}^{\delta,\varphi=1,\wedge}\rightarrow\DAlg_{\bZ_p^\un}^{\delta,\perf,\wedge}$
            is fully faithful.
        \item[{\em (c)}] The forgetful functor
            $\DAlg_{\bZ_p}^{\delta,\perf,\wedge}\rightarrow\DAlg_{\bZ_p}^{\wedge}$ is fully faithful
            with essential image those $p$-complete derived commutative $\bZ_p$-algebras $R$ such
            that $R/p$ is perfect.
    \end{enumerate}
\end{proposition}

\begin{proof}
    For part (a), the second equivalence follows from the first by taking homotopy $S^1$ fixed
    points. Part (b) follows from part (a) and Lemma~\ref{lem:unwind}. For the first equivalence of
    part (a), it is enough to handle the case when $k=\bF_p$, so consider the commutative diagram
    $$\xymatrix{
        \DAlg_{\bZ_p}^{\delta,\wedge}\ar[r]_{(-)\otimes_{\bZ_p}\bF_p}\ar[d]^{(-)_\perf}&\DAlg_{\bF_p}\ar[d]^{(-)_\perf}\\
        \DAlg_{\bZ_p}^{\delta,\perf,\wedge}\ar[r]_{(-)\otimes_{\bZ_p}\bF_p}&\DAlg_{\bF_p}^\perf.
    }$$
    Tracing around the right adjoints, we see that the right adjoint of
    $\DAlg_{\bZ_p}^{\delta,\perf,\wedge}\rightarrow\DAlg_{\bF_p}^\perf$ is given as $k\mapsto\bW(k)$.
    In particular, if $k$ is a perfect derived commutative
    $\bF_p$-algebra, then $\bW(k)$ is perfect as a derived $\delta$-ring. It is also discrete by
    Lemma~\ref{lem:discrete} if $k$ is.
    Thus, for $k$ discrete and perfect, the counit map $\bW(k)\otimes_{\bZ_p}\bF_p\rightarrow k$
    is an equivalence, which follows from the discreteness of $\bW(k)$ and the usual equivalence between perfect $p$-complete
    $\delta$-rings and perfect $\bF_p$-algebras, as in~\cite[Cor.~2.31]{prisms}.

    Now, we use the result from~\cite{mathew-mondal} or~\cite[Thm.~5.20]{brantner-campos-nuiten} to write any $k\in\DAlg_{\bF_p}^\perf$ as the
    limit of a cosimplicial diagram $k^\bullet$ where each $k^n$ is in $\CAlg_{\bF_p}^\perf$. In
    particular, $\bW(k)\we\bW(\Tot k^\bullet)\we\Tot(\bW(k^\bullet))$ and reducing modulo $p$ we
    find that $\bW(k)/p\we\Tot(k^\bullet)\we k$.

    Thus, it follows that
    $\bW\colon\DAlg_{\bF_p}^\perf\rightarrow\DAlg_{\bZ_p}^{\delta,\perf,\wedge}$ is fully faithful.
    On the other hand, given $R\in\DAlg_{\bZ_p}^{\delta,\perf,\wedge}$, the unit map $R\rightarrow
    \bW(R/p)$ is a $p$-adic equivalence as $R/p\rightarrow \bW(R/p)/p\we R/p$ is an equivalence. This
    completes the proof of part (a).

    For part (c), let $\DAlg_{\bZ_p}^{\perf,\wedge}\subseteq\DAlg_{\bZ_p}^\wedge$ be the full
    subcategory consisting of those $p$-complete derived commutative rings such that $R/p$ is
    perfect. By part (a), it is enough to show that the right adjoint to the forgetful functor
    $\DAlg_{\bZ_p}^{\delta,\perf,\wedge}\rightarrow\DAlg_{\bZ_p}^{\perf,\wedge}$ is fully faithful.
    This right adjoint is equivalent to the functor $R\mapsto\bW(R)^\perf$. So, we want to show that for
    $R\in\DAlg_{\bZ_p}^{\perf,\wedge}$ the natural map $\bW(R)^\perf\rightarrow R$ is an equivalence.

    If $R$ is a derived $p$-complete commutative ring, then we claim that
    $\bW(R)^\perf\we\bW((R/p)^\perf)\we\bW(R/p)^\perf$. Note that we do not require either $R$ or
    $\bW(R)$ to be $p$-torsion free, so that $\bW(R)/p$ and $R/p$ might have homology in degree $1$.
    However, as perfect derived commutative $\bF_p$-algebras are coconnective,
    it must be that $(R/p)^\perf$ and $\bW(R)^\perf/p\we(\bW(R)/p)^\perf$ have vanishing homology
    outside of degrees $-1$ and $0$. In particular, $(\bW(R)/p)^\perf\iso(\pi_0(\bW(R)/p))^\perf$
    and $(R/p)^\perf\iso(\pi_0(R/p))^\perf$. The map $\pi_0(\bW(R)/p)\rightarrow\pi_0(R/p)$
    is surjective and
    contains elements annihilated by Frobenius using the equation $\F\V=p$. Thus,
    $\pi_0(\bW(R)/p)\rightarrow\pi_0(R/p)$ is surjective with nilpotent kernel. It follows that the
    two perfections agree. This proves the first equivalence above and 
    a similar argument shows that $\bW((R/p)^\perf)\we\bW(R/p)^\perf$.

    Now, choose a $p$-complete cosimplicial commutative ring $R^\bullet$ such that
    $R\we\Tot(R^\bullet)$. Then, $R^\bullet/p$ is not necessarily perfect in each degree, but
    $\Tot(R^\bullet/p)$ is perfect, so $R/p\we\Tot((R^\bullet/p)^\perf)\we\Tot(R^\bullet/p)$.
    Now, we have equivalences $\bW(R^\bullet)^\perf\we \bW(R^\bullet/p)^\perf\we \bW((R^\bullet/p)^\perf)$
    of cosimplicial objects, where we appeal to~\cite[Lem.~3.2(v)]{bms1} for the first equivalence
    (whose proof in fact applies to the derived inverse limit perfection used here and not merely the classical
    inverse limit perfection)
    and limit preservation of the right adjoint $\bW$ for the second.
    In particular, this shows that $\bW(R)^\perf\we\Tot(\bW(R^\bullet))^\perf$
    is equivalent modulo $p$ to $R/p$, so the counit map $\bW(R)^\perf\rightarrow R$ is a $p$-adic
    equivalence, as desired.
\end{proof}

We perturb our notation momentarily to compare the two notions of Witt vectors constructed for
perfect derived $\bF_p$-algebras in this paper.

\begin{corollary}
    \label{cor:compare}
    Let $\bW^\delta$ denote for the moment the composition of the right adjoint of the extension of scalars functor
    $\DAlg_{\bZ_p}^{\delta,\perf,\wedge}\rightarrow\DAlg_{\bF_p}^{\perf}$ with the forgetful functor
    $\DAlg_{\bZ_p}^{\delta,\perf,\wedge}\rightarrow\DAlg_{\bZ_p}^\wedge$ and let $\bW$ denote the
    functor $\DAlg_{\bF_p}^\perf\rightarrow\DAlg_{\bZ_p}^{\wedge}$ constructed in
    Theorem~\ref{thm:nonconnective}. There is an equivalence of functors $\bW\we\bW^\delta$.
\end{corollary}

\begin{proof}
    By construction, if $k\in\DAlg_{\bF_p}^\perf$, both $\bW^\delta(k)$ and $\bW(k)$ are in
    $\DAlg_{\bZ_p}^{\perf,\wedge}\subseteq\DAlg_{\bZ_p}^\wedge$. As this $\infty$-category is
    equivalent to $\DAlg_{\bF_p}^\perf$, it is enough to show that
    $\bW^\delta(-)\otimes_{\bZ_p}\bF_p\we\bW(-)\otimes_{\bZ_p}\bF_p$. However, the left-hand functor
    is equivalent to the identity by construction of the equivalence in
    Proposition~\ref{prop:fully_faithful}(c) and the
    right-hand functor is equivalent to the identity by the argument in the second paragraph of the proof of
    Theorem~\ref{thm:nonconnective}.
\end{proof}

Derived $\delta$-rings have all colimits as do derived $\delta$-rings with a trivialization of the
associated Frobenius. It follows that for any space $X$, we can view $\bZ_p^X$ as a derived
$\delta$-ring with a trivialization of Frobenius, or $\bZ_p^{\un,X}$ as a derived $\delta$-ring.

\begin{theorem}[The $p$-adic cochain theorem]\label{thm:delta_cochains}
    \begin{enumerate}
        \item[{\em (a)}] The functor $X\mapsto \bZ_p^X$ of continuous $p$-adic cochains induces
            an equivalence
            $\Pro(\Sscr_{p\fin})^\op\we\DAlg_{\bZ_p}^{\delta,\varphi=1,\wedge}$.
        \item[{\em (b)}] The functor $X\mapsto\bZ_p^{\un,X}$ of continuous cochains with values in
            $\bZ_p^\un$ induces a fully faithful functor
            $\Pro(\Sscr_{p\fin})^\op\rightarrow\DAlg_{\bZ_p^\un}^{\delta,\perf,\wedge}\subseteq\DAlg_{\bZ_p^\un}^{\delta,\wedge}$.
    \end{enumerate}
\end{theorem}

\begin{proof}
    The result follows from Theorem~\ref{thm:derivedstone} and
    Proposition~\ref{prop:fully_faithful}, using that $\bW(\bF_p^X)\we\bZ_p^X$ and
    $\bW(\bFbar_p^X)\we\bZ_p^{\un,X}$ for any $X$.
\end{proof}

\section{Integral spaces}\label{sec:integral_spaces}

We define derived $\lambda$-rings, giving a brief exposition. The reader can follow the
ideas of Brantner--Mathew~\cite{brantner-mathew}, Raksit~\cite{raksit},
Holeman~\cite{holeman-derived}, or Kubrak--Shuklin--Zakharov~\cite{ksz}.

\begin{lemma}[$\lambda$-ring filtered monad]\label{lem:filteredlambda}
    Let $\Sym_{\bZ}^\lambda$ be the free $\lambda$-ring monad on $\Mod_{\bZ}^{\mathrm{flat}}$, the accessible
    category of flat $\bZ$-modules. Given a free module $\bZ^S$, generated by elements $x_s$ for
    $s\in S$, denote by $\Sym_{\bZ}^{\lambda,\leq
    i}\bZ^S\subseteq\Sym_{\bZ}^\lambda\bZ^S$ the subgroup generated by linear combinations of all monomials of the form
    $$\lambda^{i_1}(x_{s_1})^{e_1}\cdots\lambda^{i_m}(x_{s_m})^{e_m}$$ with $\sum_{a=1}^me_ai_a\leq i$, for
    $s_a\in S$ distinct. Then,
    \begin{enumerate}
        \item[{\em (a)}] $\Sym_{\bZ}^{\lambda,\leq i}$ determines a well-defined subfunctor of
            $\Sym_{\bZ}^\lambda$ for each $i\geq 0$;
        \item[{\em (b)}] the monad multiplication
            $\Sym_{\bZ}^\lambda\circ\Sym_{\bZ}^\lambda\rightarrow\Sym_{\bZ}^\lambda$ restricts to
            $\Sym_{\bZ}^{\lambda,\leq
            j}\circ\Sym_{\bZ}^{\lambda,\leq i}\rightarrow\Sym_{\bZ}^{\lambda,\leq ij}$;
        \item[{\em (c)}] $\Sym_{\bZ}^{\lambda,\leq i}$ is additively polynomial of degree $i$ for each
            $i\geq 0$.
    \end{enumerate}
    Thus, $\Sym_{\bZ}^{\lambda,\leq\star}$ is a filtered additively polynomial monad.
\end{lemma}

\begin{proof}
    That $\Sym_{\bZ}^{\lambda,\leq i}$ is well-defined follows from the $\lambda$-addition rule
    $\lambda^i(x+y)=\sum_{a+b=i}\lambda^i(x)\lambda^j(y)$. This implies that it is closed under
    arbitrary maps of free abelian groups and defines a subfunctor of $\Sym_{\bZ}^\lambda$. The general extension to flat $\bZ$-modules follows by
    taking filtered colimits. This proves (a).

    For (b), it is enough to show that if
    $$r=\lambda^{i_1}(x_{s_1})^{e_1}\cdots\lambda^{i_m}(x_{s_m})^{e_m}$$ with $\sum_{a=1}^me_ai_a\leq i$, for
    $s_a\in S$ distinct, representing an element of $\Sym_{\bZ}^{\lambda,\leq i}$, then
    $\lambda^j(r)\in\Sym_{\bZ}^{\lambda,\leq ij}\bZ^S$. This is computed using multiple applications of
    the $\lambda$-multiplication rule
    $$\lambda^j(xy)=P_j(\lambda^1(x),\ldots,\lambda^j(x);\lambda^1(y),\ldots,\lambda^j(y))$$ and the
    $\lambda$-composition rule
    $$\lambda^j(\lambda^i(x))=P_{j,i}(\lambda^1(x),\ldots,\lambda^{ij}(x)).$$
    The universal polynomials $P_j$ and $P_{j,i}$ are defined as follows. Consider the coefficient of $t^j$ in
    the product $$\prod_{1\leq a,b\leq j}(1+x_ay_bt),$$ which is a symmetric polynomial of bidegree
    $(j,j)$ in the $x_a$
    and $y_b$, which can therefore be expressed as a polynomial
    $P_j(e_1,\ldots,e_j;f_1,\ldots,f_j)$, where the $e_a$ and $f_b$ are the elementary symmetric
    functions in the $x_a$ and $y_b$. For $P_{j,i}$, one takes the coefficient of $t^j$ in
    $$\prod_{1\leq a_1,\ldots,a_i\leq ij}(1+x_{a_1}\cdots x_{a_i}t),$$ which is a symmetric
    homogeneous polynomial of degree $ij$, and writes it as a polynomial
    $P_{j,i}(e_1,\ldots,e_{ij})$ of the elementary symmetric polynomials of the $x_a$.
    The claim now follows for degree reasons.

    To show additive polynomiality, it is enough to show it for the graded pieces of the filtration.
    Write $\Sym_{\bZ}^{\lambda,=i}$ for the terms of weight exactly $i$. Then, there is a natural
    decomposition $$\Sym_{\bZ}^{\lambda,=i}(M\oplus
    N)\we\bigoplus_{a+b=i}\Sym_{\bZ}^{\lambda,=a}(M)\otimes\Sym_{\bZ}^{\lambda,=b}(N)$$ for each $i\geq 0$.
    Taking the fiber of this map to $\Sym_{\bZ}^{\lambda,=i}(N)$, which computes the derivative with
    respect to $M$, we find $$\bigoplus_{a+b=i,b\neq
    i}\Sym_{\bZ}^{\lambda,=a}(M)\otimes\Sym_{\bZ}^{\lambda,=b}(N),$$ which is a sum of functors in
    $N$ which are additively polynomial of degrees $\leq i-1$. Thus, the result follows by induction
    from the base case of $i=0$, which is the constant functor on $\bZ$.
\end{proof}

\begin{definition}[Derived $\lambda$-rings]
    Following~\cite[Const.4.2.19]{raksit}, the filtered additively polynomial monad
    $\Sym_{\bZ}^{\lambda,\leq\star}$ extends to a filtered excisively polynomial monad
    $\LSym_{\bZ}^{\lambda,\leq\star}$ on $\D(\bZ)$. We let $\LSym_{\bZ}^\lambda$ denote the colimit and write
    $\DAlg_{\bZ}^\lambda$ for the $\infty$-category of $\LSym_{\bZ}^\lambda$-modules in $\D(\bZ)$.
    The objects of $\DAlg_{\bZ}^\lambda$ are called derived $\lambda$-rings.
\end{definition}

\begin{remark}
    There is a forgetful functor $\DAlg_{\bZ}^\lambda\rightarrow\DAlg_{\bZ}$ which preserves all
    limits and colimits, so derived $\lambda$-rings are monadic and comonadic over
    $\DAlg_{\bZ}$. The right adjoint is by definition $\WW$, the derived big Witt vector functor. 
\end{remark}

\begin{remark}
    The full subcategory of $\DAlg_{\bZ}^\lambda$ consisting of those derived $\lambda$-rings whose
    underlying object in $\D(\bZ)$ is connective is equivalent to the $\infty$-category of animated
    commutative $\lambda$-rings, obtained by freely adjoining sifted colimits to the full
    subcategory of $\CAlg_{\bZ}^\lambda$ on the free $\lambda$-rings $\Sym^\lambda_{\bZ}\bZ^S$ for
    $S$ finite. The full subcategory of objects of $\DAlg_{\bZ}^\lambda$ with discrete underlying
    object is equivalent to the $1$-category of $\lambda$-rings. See~\cite{raksit}
    or~\cite{holeman-derived} for more details.
\end{remark}

\begin{remark}
    The $\lambda$-rings we consider are derived analogs of special $\lambda$-rings in the older
    literature.
\end{remark}

\begin{remark}[Adams operations]
    Every ordinary $\lambda$-ring is equipped with Adams operations, which amounts to an action of
    the multiplicative monoid $\bZ_{>0}$ on each $\lambda$-ring $R$. A given $k\in\bZ_{>0}$
    determines an Adams operation $\psi^k$ and the various Adams operations commute. One way of
    expressing this is that there is a map of monads $\Sym^\psi\rightarrow\Sym^\lambda$, where
    $\Sym^\psi$ is the monad of commutative rings equipped with an action of $\bZ_{>0}$.
    The free $\Sym^\psi$-algebra on a single element $x$ is, as a commutative ring, isomorphic to
    $\bZ[\psi^1(x),\psi^2(x),\psi^3(x),\psi^4(x),\ldots]$, where $\psi^1(x)=x$.
    The Adams operations on this ring are determined by $\psi^j(\psi^i(x))=\psi^{ij}(x)$ for all
    $i,j\geq 1$. The Adams ring monad is also derivable along the same lines as
    Lemma~\ref{lem:filteredlambda}; as
    a consequence, there is a map $\LSym^\psi_{\bZ}\rightarrow\LSym^\lambda_{\bZ}$. Forgetting 
    along this map induces a functor $\DAlg_{\bZ}^\psi\leftarrow\DAlg_{\bZ}^\lambda$.
\end{remark}

\begin{definition}[Perfect $\lambda$-rings]\label{def:perfectlambda}
    A derived $\lambda$-ring $R$ is perfect if $\psi^k\colon R\rightarrow R$ is an equivalence for every integer $k\geq 1$.
    It is enough to ask that $\psi^p$ is an equivalence for each prime $p$.  Let
    $\DAlg_{\bZ}^{\lambda,\perf}\subseteq\DAlg_{\bZ}^\lambda$ be the full subcategory consisting of
    the perfect $\lambda$-rings. The inclusion preserves all limits and colimits and so admits both
    left and right adjoints.
\end{definition}

\begin{remark}
    Fix a prime $p$ and a derived $\lambda$-ring $R$.
    There is a natural functor $\DAlg_{\bZ}^{\lambda}\rightarrow\DAlg_{\bF_p}$ given by
    $(-)\otimes_{\bZ}\bF_p$. If $\psi^p\colon R\rightarrow R$ is an equivalence, then
    $R\otimes_{\bZ}\bF_p$ is a perfect derived commutative $\bF_p$-algebra and hence its higher
    homotopy groups vanish. In particular, this implies that $\pi_0R$ is $p$-torsion free
    and $\pi_i(R)$ is $p$-divisible for each $i>0$. Thus, if $R$ is a perfect derived
    $\lambda$-ring, then
    $\pi_iR$ is a $\bQ$-vector space for each $i>0$ and $\pi_0R$ is torsion free.
\end{remark}

\begin{remark}[Adams and Lambda]
    The Adams operations are more than just commuting ring endomorphisms of a $\lambda$-ring $R$:
    they are commuting $\lambda$-ring endomorphisms of $R$. This follows for example
    from~\cite[Prop.~1.2]{wilkerson}. It follows (arguing for example as
    in~\cite[Const.~2.4.1]{holeman-derived}) that each derived $\lambda$-ring admits a natural action of
    $\bZ_{>0}$ and hence there is an action of $\B\bZ_{>0}$ on $\DAlg_{\bZ}^\lambda$. On
    $\DAlg_{\bZ}^{\lambda,\perf}$, this action extends to an action of $\B\bQ_{>0}$, where
    $\bQ_{>0}$ is the group completion of $\bZ_{>0}$.
\end{remark}

\begin{definition}[Binomial derived rings]
    Let
    $\DAlg_{\bZ}^{\lambda,\psi=1}=(\DAlg_{\bZ}^{\lambda,\perf})^{\h\B\bQ_{>0}}\we(\DAlg_{\bZ}^\lambda)^{\h\B\bZ_{>0}}$.
    We call the objects binomial derived rings. Reasoning as in Lemma~\ref{lem:adjoints2}, the
    forgetful functor $\DAlg_{\bZ}^{\lambda,\psi=1}\rightarrow\DAlg_{\bZ}^{\lambda,\perf}$ preserves
    all limits and adjoints and hence admits left and right adjoints.
\end{definition}

\begin{remark}
    Roughly speaking, a binomial derived ring is a derived $\lambda$-ring $R$ with compatible trivializations of the Adams
    operations $\psi^n\we\id_R$ for all $n\geq 1$.
\end{remark}

\begin{remark}
    Here is the motivation for the terminology above.
    A discrete perfect derived
    $\lambda$-ring defines an object of $\DAlg_{\bZ}^{\psi=1}$ if and only if it is a binomial ring,
    i.e., a torsion free commutative ring $R$ such that
    $$\binom{x}{n}=\frac{x(x-1)\cdots(x-n+1)}{n}\in R$$ for all $x\in R$ and all $n\geq 1$.
    This follows from a result of Elliott~\cite[Prop.~8.3]{elliott}.
\end{remark}

\begin{example}
    The rings $\bZ$, $\bZ_{(p)}$, and $\bQ$, with their unique $\lambda$-ring structures are
    binomial rings.
\end{example}

\begin{remark}
    We expect that a derived $\lambda$-ring $R$ is the data of a derived commutative ring with a
    $\bZ_{>0}$-action such that each $\psi^p$ is a (derived) lift of Frobenius
    (so that $(R,\psi^p)$ is a derived $\delta_p$-ring) and such that the $\psi^k$ are given the
    structures of derived $\delta_p$-ring maps compatibly for all $p$ and $k$.
    However, we do not have a proof at the moment.
\end{remark}

\begin{lemma}\label{cor:lambda_cases}
    For a derived $\lambda$-ring $R$, let $\DAlg_R^\lambda=(\DAlg_\bZ^\lambda)_{R/}$.
    \begin{enumerate}
        \item[{\em (i)}] There is a natural equivalence $\DAlg_{\bQ}^\lambda\we\Fun(\B
            \bZ_{>0},\DAlg_{\bQ})$, i.e., a rational derived $\lambda$-ring is the same as a derived
            commutative $\bQ$-algebra equipped with countably many commuting endomorphisms $\psi^p$.
        \item[{\em (ii)}]
            There is an equivalence
            $\DAlg_{\bZ_{(p)}}^\lambda\we\Fun(\B\bZ_{>0}',\DAlg_{\bZ_{(p)}}^\delta)$,
            where $\bZ_{>0}'$ is the multiplicative monoid of positive integers relatively
            prime-to-$p$.
    \end{enumerate}
\end{lemma}

\begin{proof}
    We give the proof of (ii), the proof of (i) being similar. The $\infty$-category
    $\DAlg_{\bZ_{(p)}}^{\lambda}$ is also monadic and comonadic over $\DAlg_{\bZ_{(p)}}$ and one can
    check that it is obtained as modules for the $\LSym_{\bZ_{(p)}}^\lambda$-monad, which is
    obtained by deriving the free $\lambda$-ring monad over $\bZ_{(p)}$. However, this monad
    controls the data of a $\delta_p$-ring together with commuting endomorphisms $\psi^\ell$ of the
    $\delta_p$-ring.
\end{proof}

\begin{remark}[Categories of models II]
    We expect that the homotopy theories of cosimplicial $\lambda$-rings, cosimplicial perfect
    $\lambda$-rings, and cosimplicial binomial rings to recover the $\infty$-categories
    $\DAlg_{\bZ}^{\lambda,\ccn}$, $\DAlg_{\bZ}^{\lambda,\perf,\ccn}$, and
    $\DAlg_{\bZ}^{\psi=1,\ccn}$, respectively.
    Ekedahl in particular has studied cosimplicial binomial rings in~\cite{ekedahl-minimal}.
\end{remark}

\begin{lemma}\label{lem:binomial_cases}
    For a binomial derived ring $R$, let $\DAlg_R^{\psi=1}=(\DAlg_\bZ^{\psi=1})_{R/}$.
    \begin{enumerate}
        \item[{\em (i)}] The forgetful functor $\DAlg_{\bQ}^{\psi=1}\rightarrow\DAlg_{\bQ}$ is an
            equivalence, i.e., a rational binomial derived ring is the same thing as a derived
            commutative $\bQ$-algebra.
        \item[{\em (ii)}] The forgetful functor
            $\DAlg_{\bZ_{(p)}}^{\psi=1}\rightarrow\DAlg_{\bZ_{(p)}}^{\delta,\varphi=1}$ is an
            equivalence.
    \end{enumerate}
\end{lemma}

\begin{proof}
    We give the proof of (i). The proof of (ii) is similar. Let $\Cscr$ be any $\infty$-category and
    let $G$ be a grouplike $\bE_2$-space. Then, the natural ``constant action'' functor
    $\Cscr\rightarrow\Fun(\B G,\Cscr)^{\h\B G}$ is an equivalence.\footnote{The functor category
    $\Fun(\B G,\Cscr)$ is the right Kan extension of the map $\ast\rightarrow\Pr^\L$ classifying
    $\Cscr$ along the map $\ast\rightarrow\B^2 G$. Taking $\B G$-invariants corresponds to further
    right Kan extending $\B^2 G\rightarrow\Pr^\L$ along $\B^2 G\rightarrow\ast$. But, the
    composition is the identity. Indeed, the two Kan extensions are the right adjoints to the
    pullback functors $\Fun(\ast,\Pr^\L)\rightarrow\Fun(\B^2
    G,\Pr^\L)\rightarrow\Fun(\ast,\Pr^\L)$ whose composition is the identity.} Now, use
    Corollary~\ref{cor:lambda_cases}.
\end{proof}

\begin{lemma}[The binomial fracture square]\label{lem:binomial_fracture}
    The $\infty$-category $\DAlg_\bZ^{\psi=1}$ is equivalent to the pullback
    $$\xymatrix{
        \DAlg_\bZ^{\psi=1}\ar[r]\ar[d]&\DAlg_{\bQ}^{\Delta^1}\ar[d]^{\text{{\em codomain}}}\\
        \prod_p\DAlg_{\bZ_p}^{\delta,\varphi=1,\wedge}\ar[r]&\DAlg_{\bQ}.
    }$$ It follows that for any $R,S\in\DAlg_{\bZ}^{\psi=1}$, there is a pullback square
    $$\xymatrix{
        \Map_{\DAlg_{\bZ}^{\psi=1}}(R,S)\ar[r]\ar[d]&\Map_{\DAlg_{\bQ}}(R_\bQ,S_\bQ)\ar[d]\\
        \prod_p\Map_{\DAlg_{\bZ_p}^{\delta,\varphi=1,\wedge}}(R_p^\wedge,S_p^\wedge)\ar[r]&\Map_{\DAlg_{\bQ}}(R_{\bQ},(\prod_p S_p^\wedge)_{\bQ}).
    }$$
\end{lemma}

\begin{proof}
    The proof uses the standard fracture square for objects of $\D(\bZ)$, the fact that
    $\DAlg_{\bZ}^{\psi=1}$ is closed under products, completions, and tensor products, and
    Lemma~\ref{lem:binomial_cases}. The second part is a diagram chase for mapping spaces in
    pullback $\infty$-categories.
\end{proof}

\begin{definition}[$\bZ$-finite spaces]
    Say that a space $X\in\Sscr$ is nilpotent if for each $x\in X$ the group
    $\pi_1(X,x)$ is nilpotent and acts nilpotently on $\pi_i(X,x)$ for each $i\geq 2$. Let
    $\Sscr_{\nil}\subseteq\Sscr$ be the full subcategory of nilpotent spaces.
    We will say that $X$ is $\bZ$-finite if it is nilpotent, $n$-truncated for some $n$, and has finitely many components
    $Y\subseteq X$, each of which satisfies the following conditions:
    \begin{enumerate}
        \item[(1)] $\pi_1(Y)$ is finitely presented;
        \item[(2)] $\pi_i(Y)$ is a finitely presented abelian group for $i\geq 2$.
    \end{enumerate}
    Let $\Sscr_{\bZ\fin}\subseteq\Sscr_\nil$ be the full subcategory of $\bZ$-finite spaces. Say that
    a nilpotent space is $\bZ$-finite type if it satisfies the conditions of being $\bZ$-finite
    except that it is not required to be truncated; denote the full subcategory of $\Sscr$
    consisting of such spaces by $\Sscr_{\bZ\ft}$. Thus, we have inclusions
    $$\Sscr_{\bZ\fin}\subseteq\Sscr_{\bZ\ft}\subseteq\Sscr_\nil.$$
\end{definition}

\begin{construction}
    As $\Sscr_{\bZ\fin}$ is a full subcategory of $\Sscr$, we can define a functor
    $\Pro(\Sscr_{\bZ\fin})\rightarrow\Sscr$ by right Kan extension, also known as taking the limit
    of a pro-$\bZ$-finite space. This is the right adjoint of an
    adjunction $$\Sscr\rightleftarrows\Pro(\Sscr_{\bZ\fin}),$$ where the left adjoint is
    pro-finite nilpotent completion $X\mapsto X^\nil$. If $X\in\Sscr_{\bZ\ft}$, then $X^\nil$ is
    equivalent to the pro-$\bZ$-finite space $\tau_{\leq\star}X$. As $X\we\lim\tau_{\leq\star}X$,
    the unit of the adjunction is an equivalence for $X\in\Sscr_{\bZ\ft}$ so taking Postnikov
    towers gives a fully faithful functor $\Sscr_{\bZ\ft}\rightarrow\Pro(\Sscr_{\bZ\fin})$.
\end{construction}

\begin{example}
    If $X\in\Sscr^\omega$ is simply connected, then $X$ is in $\Sscr_{\bZ\ft}$. It is in
    $\Sscr_{\bZ\fin}$ if and only if it is truncated.
\end{example}

\begin{theorem}[The integral cochain theorem]\label{thm:integral_cochain}
    The natural functor $\Sscr_{\bZ\ft}^\op\rightarrow\DAlg_{\bZ}^{\psi=1}$ is fully faithful. More
    generally, for any space $Y\in\Sscr$ and $X\in\Pro(\Sscr_{\bZ\fin})$, the natural map
    $$\Map_\Sscr(Y,\lim
    X)\we\Map_{\Pro(\Sscr_{\bZ\fin})}(Y^\nil,X)\rightarrow\Map_{\DAlg_{\bZ}^{\psi=1}}(\bZ^X,\bZ^Y)$$
    is an equivalence, where $\bZ^X$ denotes continuous cochains in $\bZ$.
\end{theorem}

\begin{proof}
    As in the proof of Theorem~\ref{thm:derivedstone}, we can pull $Y$ out and reduce to the case
    where $Y\we\ast$. We also have that if $X$ is in $\Sscr_{\bZ\ft}$, then the natural map
    $\colim_n\bZ^{\tau_{\leq n}X}\rightarrow\bZ^X$ is an equivalence. Moreover, the profinite completion
    of $\bZ^X$ is $\widehat{\bZ}^X$, the rationalization of $\bZ^X$ is $\bQ^X$, and the rationalization of
    the profinite completion is $\bA^X$, where $\bA=(\prod_p\bZ_p)\otimes_{\bZ}\bQ$ is the ring of
    finite ad\`eles. By Lemma~\ref{lem:binomial_fracture}, there is a pullback square
    $$\xymatrix{
        \Map_{\DAlg_{\bZ}^{\psi=1}}(\bZ^X,\bZ)\ar[r]\ar[d]&\Map_{\DAlg_{\bQ}}(\bQ^X,\bQ)\ar[d]\\
        \prod_p\Map_{\DAlg_{\bZ_p}^{\delta,\varphi=1,\wedge}}(\bZ_p^X,\bZ_p)\ar[r]&\Map_{\DAlg_{\bQ}}(\bQ^X,\bA),
    }$$
    which can be rewritten as
    $$\xymatrix{
        \Map_{\DAlg_{\bZ}^{\psi=1}}(\bZ^X,\bZ)\ar[r]\ar[d]&X_\bQ\ar[d]\\
        \prod_p X_{\bF_p}\ar[r]&X_{\bA}
    }$$
    using Theorems~\ref{thm:sullivan} and~\ref{thm:delta_cochains}.
    The map from $X$ to the pullback is an equivalence by the arithmetic square for nilpotent spaces
    (see for example~\cite[Prop.~3.20]{sullivan-geometric}).
\end{proof}

\begin{remark}
    To\"en describes in~\cite{toen-schematisation} the homotopy sheaves of the affine stack associated to a simply connected
    space: $$\pi_i(\Spec\R\Gamma(\bZ^X))\iso\pi_iX\otimes_{\bZ}\bH,$$ where $\bH$ is the group
    scheme $\Spec\bZ\binom{x}{\bullet}$. This result can be deduced from
    Theorem~\ref{thm:integral_cochain} by using that the cofree binomial ring on $R$ is $\bH(R)$.
\end{remark}

\section{Synthetic spaces}\label{sec:synthetic}

Our last model for $p$-adic spaces is via the synthetic cochain functor.
Let $\bS_\syn^\un$ be the $p$-complete even filtration on $\bS_p^\un$. Given a space, we can form
$\bS_\syn^{\un,X}$, which we will call the unramified $p$-adic synthetic cochains on $X$. It naturally admits
the structure of an $\bE_\infty$-ring in complete filtered $\bS_\syn^\un$-module spectra.
While not relevant for the theorem below, we note
that $\bS_\syn^{\un,X}\we\F^\star_\syn\widehat{\bS\bW}(\bFbar_p^X)$ in the context of
Theorem~\ref{thm:nonconnective}.

\begin{theorem}\label{thm:synthetic}
    The functor
    $\bS_\syn^{\un,(-)}\colon\Sscr_{p\ft}^\op\rightarrow\CAlg(\FDhat(\bS_\syn^\un)_p^\wedge)$ is fully
    faithful.
\end{theorem}

Recall that Yuan proves that $X\mapsto\bS_p^{\un,X}$ is fully faithful when restricted to simply
connected $p$-complete finite spaces, i.e., those simply connected spaces which are built out of
finitely many $p$-complete spheres. The theorem says that one can extend the range of this theorem
at the cost of working in a more complicated $\infty$-category.

\begin{proof}
    The standard argument lets us reduce to proving that for all $X\in\Sscr_{p\ft}$, the natural
    map $X\rightarrow\Map_{\CAlg(\FD(\bS_\syn^\un)_p^\wedge)}(\bS_\syn^{\un,X},\bS_\syn^\un)$ is an equivalence.
    Using the Rees construction, we can realize $\FD(\bS_\syn^\un)_p^\wedge$, the $\infty$-category of $p$-complete synthetic 
    spectra over $\bS_\syn^\un$, as $\Mod_{\Rees(\bS_\syn^\un)}\Gr\D(\bS_p)_p^\wedge$,
    the $\infty$-category of modules over the
    Rees algebra of the unramified synthetic sphere spectrum in $p$-complete graded spectra.
    Similarly, the Rees algebra construction induces an equivalence
    $\Rees\colon\CAlg(\FD(\bS_\syn^\un)_p^\wedge)\we\CAlg(\Gr\D(\Rees(\bS_\syn^\un))_p^\wedge)$.
    The Rees algebra has an element $t$ in weight $-1$ and homological degree $0$ so that
    $$\bS_\gr^\un=\gr^\star\bS_\syn^\un\we\cofib(\bS_\syn^\un(-1)\xrightarrow{t}\bS_\syn^\un)\we\cofib(t).$$
    In particular, $\bS_\gr^\un$ is a dualizable object of $\FD(\bS_\syn^\un)$. Let
    $\bS_\gr^{\un,\bullet}$ be the \v{C}ech complex of $\Rees(\bS_\syn^\un)\rightarrow\bS_\gr^\un$
    in $p$-complete graded spectra. As the $\bS_\gr^\un$-complete objects correspond to the complete
    filtrations, we obtain an equivalence
    $$\CAlg(\FDhat(\bS_p^\un)_p^\wedge)\we\Tot(\CAlg(\Gr\D(\bS_\gr^{\un,\bullet})_p^\wedge))$$ using~\cite[Thm.~2.30]{MNN17}.\footnote{The cited theorem
    does not in fact apply to $p$-complete synthetic spectra as the unit is not compact and is
    moreover about module categories and not $\infty$-categories of $\bE_\infty$-algebras. However,
    the first problem
    can be fixed by working in all synthetic spectra and applying the argument to the ring
    spectrum $\bS_\gr^\un/p^m$ for $m$ sufficiently large so that the quotient admits the structure
    of an $\bE_1$-ring spectrum~\cite{burklund-moore}. The second problem is fixed by noting that
    the limit is in fact of symmetric monoidal stable $\infty$-categories and hence of
    $\infty$-operads. As the $\infty$-categories of commutative algebras are given by
    $\infty$-operad maps out of the $\bE_\infty$-operad, it follows that one obtains a limit diagram
    at the level of $\bE_\infty$-algebra $\infty$-categories as well.}
    We have
    $$\gr^\star\bS_\syn^{\un,X}\we\Rees(\bS_\syn^{\un,X})\otimes_{\Rees(\bS_\syn^{\un})}\bS_\gr^{\un}\we\bZ_p^{\un,X}\otimes_{\bZ_p^\un}\bS_\gr^{\un}$$
    where $\bZ_p^\un$ and $\bZ_p^{\un,X}$ are concentrated in weight $0$
    as $\bS_\gr^\un$ is perfect over $\bZ_p^\un$ in each weight.
    It follows that for any cosimplicial degree $n$ there is an equivalence
    $$\Rees(\bS_\syn^{\un,X})\otimes_{\Rees(\bS_\syn^\un)}\bS_\gr^{\un,n}\we\bZ_p^{\un,X}\otimes_{\bZ_p^\un}\bS_\gr^{\un,n},$$
    which depends on the choice of a $\bZ_p^\un$-algebra structure on $\bS_\gr^{\un,n}$.
    Thus, for each fixed $n$, we have 
    \begin{align*}
        \Map_{\CAlg(\Gr\D(\bS_\gr^{\un,n})_p^\wedge)}\left(\Rees(\bS_\syn^{\un,X})\otimes_{\Rees(\bS_\syn^\un)}\bS_\gr^{\un,n},\bS_\gr^{\un,n}\right)
        &\we\Map_{\CAlg(\Gr\D(\bZ_p)_p^\wedge)}(\bZ_p^{\un,X},\bS_\gr^{\un,n})\\
        &\we\Map_{\CAlg(\D(\bZ_p)_p^\wedge)}(\bZ_p^{\un,X},\gr^0(\bS_\gr^{\un,n})),
    \end{align*}
    where the last equivalence follows from the adjunction
    $$\ins^0\colon\D(\bZ_p)_p^\wedge\rightleftarrows\Gr\D(\bZ_p)_p^\wedge\colon\ev^0.$$
    Now, we claim that $\gr^0(\bS_\gr^{\un,n})$ is a connective $\bE_\infty$-ring spectrum with
    $\pi_0\iso\bZ_p^\un$.
    This is clear if $n=0$. We give the argument for $n=1$, the cases when $n>1$ being identical but
    notationally heavier. To compute $\gr^0(\bS_\gr^\un\otimes_{\Rees(\bS_\syn^\un)}\bS_\gr^\un)$ one
    can compute the geometric realization of the bar construction
    $\bS_\gr^\un\otimes\Rees(\bS_\syn^\un)^{\otimes\bullet}\otimes\bS_\gr^\un$.
    In each simplicial degree $m$, one uses the direct sum formula for tensor products of graded
    objects to compute
    $$\gr^0(\bS_\gr^\un\otimes\Rees(\bS_\syn^\un)^{\otimes
    m}\otimes\bS_\gr^\un)\we\bigoplus_{i+j+k_1+\cdots+k_m=0}\gr^i\bS_\syn\otimes\F^{k_1}\bS_\syn\otimes\cdots\otimes\F^{k_m}\bS_\syn\otimes\gr^j\bS_\syn,$$
    where all tensor products are over $\bS_p^\un$.
    If any of $i,j,k_1,\ldots,k_m$ is positive, then the resulting term is connected as
    $\F^k\bS_\syn$ is $k$-connective for all $k\geq 0$. So, the only contribution to $\pi_0$
    of the direct sum is when $i=j=k_1=\cdots=k_m=0$, in which case the tensor product is equivalent
    to $\bZ_p^\un\otimes_{\bS_p^\un}\bZ_p^\un$. This completes the proof of the claim.

    As the $p$-complete $\bE_\infty$-cotangent complex of $\bZ_p^{\un,X}$ vanishes
    by~\cite[Prop.~7.6]{yuan-integral}, obstruction
    theory implies that for each $n$ there is an equivalence
    $$\Map_{\bZ_p^\un}(\bZ_p^{\un,X},\gr^0(\bS_\gr^{\un,n}))\we\Map_{\bZ_p^\un}(\bZ_p^{\un,X},\bZ_p^{\un})\we\Map_{\bFbar_p}(\bFbar_p^X,\bFbar_p)\we
    X,$$
    where the deformation theory is as in the proof of ~\cite[Cor.~7.6.1]{yuan-integral}
    and the final equivalence is thanks to the main theorem of~\cite{mandell-padic}.
    This shows that the natural map from $X$ to each term in the cosimplicial diagram of mapping spaces is
    an equivalence; thus the cosimplicial diagram is equivalent to the constant diagram on $X$, so
    its limit is again $X$.
\end{proof}

\begin{remark}
    Little of the argument above is unique to the sphere spectrum. It would work for
    $X\mapsto(\F^\star R)^X$ in
    $\mfrak$-complete filtered $\F^\star R$-module spectra where $\F^\star R$ is a complete filtered
    $\bE_\infty$-ring spectrum with for which $\gr^0R$
    is a (discrete) complete regular local ring with maximal ideal $\mfrak$ and algebraically closed residue field of characteristic
    $p>0$, where each $\gr^iR$ is a perfect $\gr^0R$-module spectrum, where $\gr^iR\we 0$ for
    $i<0$, and where $\F^i R$ is connective for each $i\in\bZ$.
\end{remark}

\begin{remark}
    \label{rem:tcs}
    Yuan asks in~\cite[Que.~7.13]{yuan-integral} what the space of $\bE_\infty$-maps $\bS_p^{\B
    C_p}\rightarrow\bS_p$ is, the hope being that it might be still usable to recover $\B C_p$.
    However, replacing $\bS_p$ with $\bS_p^\un$, one sees that there are at least two distinct
    homotopy classes of $\bE_\infty$-maps $\bS_p^{\un,\B C_p}\rightarrow\bS_p^\un$, one being
    given by the canonical map obtained by inclusion of fixed points and the other being
    obtained by $\bS_p^{\un,\B C_p}\we\bS_p^{\un,\h C_p}\rightarrow\bS_p^{\un,\t
    C_p}\we\bS_p^\un$, where we give $\bS_p^\un$ the trivial action. That these are distinct
    follows for example from the analysis of $\TC$ of the sphere spectrum
    (see~\cite{blumberg-mandell-homotopy}).
\end{remark}

\section{Spherical $\lambda$-rings}

We conclude the paper with a construction of new $\bE_\infty$-ring spectra.

\begin{definition}[Perfect commutative rings]
    A commutative ring $R$ is perfect if $R/p$, the derived reduction mod $p$, is a perfect $\bF_p$-algebra for all primes $p$.
\end{definition}

\begin{definition}[Spherical lifts of perfect commutative rings]\label{def:spherical_lambda}
    Suppose that $R$ is a perfect commutative ring. Let $\bS_R$ be the pullback
    \begin{equation}
        \label{eq:spherical_lambda}
        \begin{gathered}\xymatrix{
        \bS_R\ar[r]\ar[d]&R\ar[r]\ar[d]&R_\bQ\ar[d]\\
        \prod_p\bS\bW(R/p)\ar[r]^{\pi_0}&\prod_pR_p^\wedge\ar[r]&R_\bA
    }\end{gathered}\end{equation}
    of either the left commutative square or equivalently the outer square.
    This construction defines a functor $\CAlg_{\bZ}^{\perf}\rightarrow\CAlg(\Sp)$.
\end{definition}

\begin{proposition}
    \label{prop:spherical_binomial}
    The functor $\bS_{(-)}\colon\CAlg_{\bZ}^\perf\rightarrow\CAlg(\Sp)$ is fully faithful and admits
    the structure of a symmetric monoidal functor. If
    $R\in\CAlg_{\bZ}^\perf$, then $\bS_R$ has the following properties:
    \begin{enumerate}
        \item[{\em (a)}] $\pi_0\bS_R\iso R$;
        \item[{\em (b)}] $\bS_R$ is flat over $\bS$:
            $\pi_i\bS\otimes_{\bZ}\pi_0\bS_R\iso\pi_i\bS_R$;
        \item[{\em (c)}] the outer commutative square in~\eqref{eq:spherical_lambda} is equivalent
            to the arithmetic fracture square of $\bS_R$.
    \end{enumerate}
\end{proposition}

\begin{proof}
    Fully faithfulness follows by using that for $R$ and $T$ in $\CAlg_{\bZ}^\perf$ we have a pullback square
    $$\xymatrix{
        \Map_{\bS}(\bS_R,\bS_T)\ar[r]\ar[d]&\Map_{\bS}(\bS_R,T_\bQ)\ar[d]\\
        \prod_p\Map_{\bS}(\bS_R,\bS\bW(T/p))\ar[r]&\Map_\bS(\bS_R,T_\bA),
    }$$
    which can be identified with the corresponding pullback square computing $\Map_{\bZ}(R,T)$,
    using the fact that
    $$\Map_{\bS}(\bS_R,\bS\bW(T/p))\we\Map_{\bS}(\bS\bW(R/p),\bS\bW(T/p))\we\Map_{\bF_p}(R/p,T/p)\we\Map_{\bZ}(R_p^\wedge,T_p^\wedge),$$
    where the first equivalence follows by $p$-completeness, the second
    by Corollary~\ref{cor:sw_fully_faithful}, and the third
    by perfectness (see Proposition~\ref{prop:fully_faithful}).

    For symmetric monoidality, note that the symmetric monoidal structure on both sides is the
    cocartesian one. Thus, it is enough to check that $\bS_{(-)}$ preserves coproducts. If
    $R,T\in\CAlg_{\bZ}^\perf$, then $R\otimes_{\bZ}T$ is again perfect.\footnote{Note that if
    $R\in\CAlg_{\bZ}^\perf$, then it is torsion free.} The natural map
    $\bS_R\otimes_{\bS}\bS_T\rightarrow\bS_{R\otimes_{\bZ}T}$ is a map of flat $\bE_\infty$-rings
    over $\bS$ and it induces an isomorphism on $\pi_0$ by construction. Thus, it is an equivalence,
    as desired.

    Property (a) follows by computing the Mayer--Vietoris sequence for the left pullback square
    in~\eqref{eq:spherical_lambda}. Property (b) follows because, for $i\geq 1$,
    $$\pi_i\bS_R\iso\pi_i\prod_p\bS\bW(R/p)\iso\prod_p
    \pi_i\bS\otimes_{\bZ}\bW(R/p)\iso\pi_i\bS\otimes_\bZ\prod_p\bW(R/p)\iso\pi_i\bS\otimes_{\bZ} R,$$
    where we use that each stable homotopy group is finitely generated torsion.
    Property (c) follows because $(\bS_R)_\bQ\we R_\bQ$ by (b).
\end{proof}

\begin{example}
    The monoid ring $\bS[\bQ_{\geq 0}]$ arises as $\bS_{\bZ[\bQ_{\geq 0}]}$.
\end{example}

\begin{example}
    Let $\bZ\binom{x}{\bullet}\subseteq\bQ[x]$ denote the ring of integer-valued polynomials.
    By Newton, these are the finite sums $\sum_i a_i\binom{x}{i}$, where $\binom{x}{i}$ is the
    Newton polynomial $\tfrac{x(x-1)\cdots(x-n+1)}{n!}$. This is a binomial ring and hence each
    $p$-completion is perfect, so it is a perfect commutative ring. Write $\bS\binom{x}{\bullet}$
    for its spherical lift $\bS_{\bZ\binom{x}{\bullet}}$. Like the monoid ring $\bS[t]=\bS[\bN]$,
    this is a flat $\bE_\infty$-ring whose rationalization is a polynomial ring in one generator
    over $\bQ$, however it has a very different flavor. For example, one can check using the
    $\bE_\infty$-Frobenius with respect to a prime $p$ that there is no map $\bS[t]\rightarrow\bS\binom{x}{\bullet}$ sending $t$
    to $x$: indeed Frobenius sends $t$ to $t^p$ but $x$ to $x$.
\end{example}

\begin{remark}[The Hilbert additive group]
    There is a commutative ring scheme structure on $\bH=\Spec\bZ{\binom{x}{\bullet}}$, which To\"en calls the
    Hilbert additive group in~\cite{toen-schematisation} (and is related to the filtered circle as
    constructed in~\cite{moulinos-robalo-toen}). This is encoded in a commutative coring structure on $\bZ\binom{x}{\bullet}$. By
    symmetric monoidality of $\bS_{(-)}\colon\CAlg_{\bZ}^\perf\rightarrow\CAlg(\Sp)$, there is a
    co-$\bE_\infty$-ring structure on $\bS\binom{x}{\bullet}$ as well, so that mapping out of
    $\bS\binom{x}{\bullet}$ produces a functor with values in $\bE_\infty$-rings, which is a
    spherical lift of $\bH$.
\end{remark}

\begin{remark}[Spherical lifts of perfect $\lambda$-rings]
    Suppose that $R\in\CAlg_R^\lambda$ is a perfect $\lambda$-ring, meaning that each $\psi^p$ is an equivalence.
    Then, it is in particular perfect as a commutative ring and we can form $\bS_R$ as above. In
    this case the Adams operations canonically lift to $\bS_R$ as they act on each vertex in the
    pullback squares in~\eqref{eq:spherical_lambda}.
\end{remark}

\begin{conjecture}
    \label{conj:synthetic_binomial}
    By making a synthetic analogue of Definition~\ref{def:spherical_lambda} and studying an
    $\infty$-category of synthetic $\lambda$-rings with appropriate trivializations of all Frobenii,
    we expect that Theorem~\ref{thm:synthetic} can be refined to give integral models of finite type
    nilpotent spaces.
\end{conjecture}

\small
\bibliographystyle{amsplain}
\bibliography{transmutation}

\providecommand{\bysame}{\leavevmode\hbox to3em{\hrulefill}\thinspace}
\providecommand{\MR}{\relax\ifhmode\unskip\space\fi MR }
\providecommand{\MRhref}[2]{%
  \href{http://www.ams.org/mathscinet-getitem?mr=#1}{#2}
}
\providecommand{\href}[2]{#2}
\begin{thebibliography}{10}

\bibitem{artin-mazur}
M.~Artin and B.~Mazur, \emph{Etale homotopy}, Lecture Notes in Mathematics,
  vol. 100, Springer-Verlag, Berlin, 1986, Reprint of the 1969 original.
  \MR{883959}

\bibitem{bhatt-fgauges}
Bhargav Bhatt, \emph{Prismatic {F}-gauges}, available at
  \url{https://www.math.ias.edu/\~bhatt/teaching/mat549f22/lectures.pdf}
  (version accessed 10 April 2023).

\bibitem{bms1}
Bhargav Bhatt, Matthew Morrow, and Peter Scholze, \emph{Integral {$p$}-adic
  {H}odge theory}, Publ. Math. Inst. Hautes \'{E}tudes Sci. \textbf{128}
  (2018), 219--397. \MR{3905467}

\bibitem{bhatt-scholze-witt}
Bhargav Bhatt and Peter Scholze, \emph{Projectivity of the {W}itt vector affine
  {G}rassmannian}, Invent. Math. \textbf{209} (2017), no.~2, 329--423.
  \MR{3674218}

\bibitem{prisms}
\bysame, \emph{Prisms and prismatic cohomology}, Ann. of Math. (2) \textbf{196}
  (2022), no.~3, 1135--1275. \MR{4502597}

\bibitem{blomquist-harper}
Jacobson Blomquist and John Harper, \emph{Integral chains and {B}ousfield-{K}an
  completion}, Homology Homotopy Appl. \textbf{21} (2019), no.~2, 29--58.
  \MR{3911945}

\bibitem{blumberg-mandell-homotopy}
Andrew~J. Blumberg and Michael~A. Mandell, \emph{The homotopy groups of the
  algebraic {$K$}-theory of the sphere spectrum}, Geom. Topol. \textbf{23}
  (2019), no.~1, 101--134. \MR{3921317}

\bibitem{bousfield-gugenheim}
A.~K. Bousfield and V.~K. A.~M. Gugenheim, \emph{On {PL} de {R}ham theory and
  rational homotopy type}, Mem. Amer. Math. Soc. \textbf{8} (1976), no.~179,
  ix+94. \MR{425956}

\bibitem{brantner-campos-nuiten}
Lukas Brantner, Ricardo Campos, and Joost Nuiten, \emph{{PD} operads and
  explicit partition {L}ie algebras}, arXiv preprint arXiv:2104.03870 (2021).

\bibitem{brantner-mathew}
Lukas Brantner and Akhil Mathew, \emph{Deformation theory and partition {L}ie
  algebras}, arXiv preprint arXiv:1904.07352 (2019).

\bibitem{burklund-moore}
Robert Burklund, \emph{Multiplicative structures on {M}oore spectra}, arXiv
  preprint arXiv:2203.14787 (2022).

\bibitem{bsy}
Robert Burklund, Tomer Schlank, and Allen Yuan, \emph{The chromatic
  {N}ullstellensatz}, arXiv preprint arXiv:2207.09929 (2022).

\bibitem{carmeli}
Shachar Carmeli, \emph{On the strict {P}icard spectrum of commutative ring
  spectra}, arXiv preprint arXiv:2208.03073 (2022).

\bibitem{ekedahl-minimal}
Torsten Ekedahl, \emph{On minimal models in integral homotopy theory}, vol.~4,
  2002, The Roos Festschrift volume, 1, pp.~191--218. \MR{1918189}

\bibitem{elliott}
Jesse Elliott, \emph{Binomial rings, integer-valued polynomials, and
  {$\lambda$}-rings}, J. Pure Appl. Algebra \textbf{207} (2006), no.~1,
  165--185. \MR{2244389}

\bibitem{gikr}
Bogdan Gheorghe, Daniel Isaksen, Achim Krause, and Nicolas Ricka,
  \emph{{$\Bbb{C}$}-motivic modular forms}, J. Eur. Math. Soc. (JEMS)
  \textbf{24} (2022), no.~10, 3597--3628. \MR{4432907}

\bibitem{goerss-padic}
Paul Goerss, \emph{Simplicial chains over a field and {$p$}-local homotopy
  theory}, Math. Z. \textbf{220} (1995), no.~4, 523--544. \MR{1363853}

\bibitem{grothendieck-pursuing}
Alexander Grothendieck, \emph{Pursuing stacks}, arXiv preprint arXiv:2111.01000
  (2021).

\bibitem{hrw}
Jeremy Hahn, Arpon Raksit, and Dylan Wilson, \emph{A motivic filtration on the
  topological cyclic homology of commutative ring spectra}, arXiv preprint
  arXiv:2206.11208 (2022).

\bibitem{holeman-derived}
Adam Holeman, \emph{Derived {$\delta$}-rings and relative prismatic
  cohomology}, arXiv preprint arXiv:2303.17447 (2023).

\bibitem{horel}
Geoffroy Horel, \emph{Binomial rings and homotopy theory}, arXiv preprint
  arXiv:2211.02349 (2022).

\bibitem{ksz}
Dmitry Kubrak, Georgii Shuklin, and Alexander Zakharov, \emph{Derived binomial
  rings {I}: integral {B}etti cohomology of log schemes}, arXiv preprint
  arXiv:2308.01110 (2023).

\bibitem{kriz}
Igor K\v{r}\'{\i}\v{z}, \emph{{$p$}-adic homotopy theory}, Topology Appl.
  \textbf{52} (1993), no.~3, 279--308. \MR{1243609}

\bibitem{htt}
Jacob Lurie, \emph{Higher topos theory}, Annals of Mathematics Studies, vol.
  170, Princeton University Press, Princeton, NJ, 2009. \MR{2522659}

\bibitem{dag13}
\bysame, \emph{Derived algebraic geometry {XIII}: Rational and p-adic homotopy
  theory}, available at
  \url{https://www.math.ias.edu/~lurie/papers/DAG-XIII.pdf} (version dated 15
  December 2011).

\bibitem{ha}
\bysame, \emph{Higher algebra}, available at
  \url{http://www.math.harvard.edu/~lurie/papers/HA.pdf}, version dated 18
  September 2017.

\bibitem{lurie-elliptic-2}
\bysame, \emph{Elliptic cohomology {II}: orientations}, available at
  \url{https://www.math.ias.edu/~lurie/papers/Elliptic-II.pdf} (version dated
  26 April 2018).

\bibitem{dag8}
\bysame, \emph{Derived algebraic geometry {VIII}: quasi-coherent sheaves and
  {T}anaka duality theorems}, available at
  \url{https://www.math.ias.edu/~lurie/papers/DAG-VIII.pdf} (version dated 5
  November 2011).

\bibitem{dag12}
\bysame, \emph{Derived algebraic geometry {XII}: Proper morphisms, completions,
  and the {G}rothendieck existence theorem}, available at
  \url{https://www.math.ias.edu/~lurie/papers/DAG-XIII.pdf} (version dated 8
  November 2011).

\bibitem{mandell-padic}
Michael Mandell, \emph{{$E_\infty$} algebras and {$p$}-adic homotopy theory},
  Topology \textbf{40} (2001), no.~1, 43--94. \MR{1791268}

\bibitem{mandell-cochains}
\bysame, \emph{Cochains and homotopy type}, Publ. Math. Inst. Hautes \'{E}tudes
  Sci. (2006), no.~103, 213--246. \MR{2233853}

\bibitem{mathew-mondal}
Akhil Mathew and Shubhodip Mondal, \emph{Affine stacks and derived rings},
  forthcoming (2023).

\bibitem{MNN17}
Akhil Mathew, Niko Naumann, and Justin Noel, \emph{Nilpotence and descent in
  equivariant stable homotopy theory}, Adv. Math. \textbf{305} (2017),
  994--1084.

\bibitem{mondal-reinecke}
Shubhodip Mondal and Emanuel Reinecke, \emph{Unipotent homotopy theory of
  schemes}, arXiv preprint arXiv:2302.10703 (2023).

\bibitem{moulinos-robalo-toen}
Tasos Moulinos, Marco Robalo, and Bertrand To\"{e}n, \emph{A universal
  {H}ochschild-{K}ostant-{R}osenberg theorem}, Geom. Topol. \textbf{26} (2022),
  no.~2, 777--874. \MR{4444269}

\bibitem{pstragowski-even}
Piotr Pstr{\k{a}}gowski, \emph{Perfect even modules and the even filtration},
  arXiv preprint arXiv:2304.04685 (2023).

\bibitem{pstragowski-synthetic}
\bysame, \emph{Synthetic spectra and the cellular motivic category}, Invent.
  Math. \textbf{232} (2023), no.~2, 553--681. \MR{4574661}

\bibitem{quigley-shah}
JD~Quigley and Jay Shah, \emph{On the parametrized {T}ate construction and two
  theories of real {$p$}-cyclotomic spectra}, arXiv preprint arXiv:1909.03920
  (2019).

\bibitem{quillen-rational}
Daniel Quillen, \emph{Rational homotopy theory}, Ann. of Math. (2) \textbf{90}
  (1969), 205--295. \MR{258031}

\bibitem{raksit}
Arpon Raksit, \emph{Hochschild homology and the derived de {R}ham complex
  revisited}, arXiv preprint arXiv:2007.02576 (2020).

\bibitem{rivera-wierstra-zeinalian}
Manuel Rivera, Felix Wierstra, and Mahmoud Zeinalian, \emph{The simplicial
  coalgebra of chains determines homotopy types rationally and one prime at a
  time}, Trans. Amer. Math. Soc. \textbf{375} (2022), no.~5, 3267--3303.
  \MR{4402661}

\bibitem{lucio}
Victor Roca~i Lucio, \emph{Higher {L}ie theory in positive characteristic},
  arXiv preprint arXiv:2306.07829 (2023).

\bibitem{scholze-perfectoid}
Peter Scholze, \emph{Perfectoid spaces}, Publ. Math. Inst. Hautes \'{E}tudes
  Sci. \textbf{116} (2012), 245--313. \MR{3090258}

\bibitem{stringall}
R.~W. Stringall, \emph{The categories of {$p$}-rings are equivalent}, Proc.
  Amer. Math. Soc. \textbf{29} (1971), 229--235. \MR{276153}

\bibitem{sullivan-infinitesimal}
Dennis Sullivan, \emph{Infinitesimal computations in topology}, Inst. Hautes
  \'{E}tudes Sci. Publ. Math. (1977), no.~47, 269--331 (1978). \MR{646078}

\bibitem{sullivan-geometric}
\bysame, \emph{Geometric topology: localization, periodicity and {G}alois
  symmetry}, $K$-Monographs in Mathematics, vol.~8, Springer, Dordrecht, 2005,
  The 1970 MIT notes, Edited and with a preface by Andrew Ranicki. \MR{2162361}

\bibitem{toen-affines}
Bertrand To\"{e}n, \emph{Champs affines}, Selecta Math. (N.S.) \textbf{12}
  (2006), no.~1, 39--135. \MR{2244263}

\bibitem{toen-schematisation}
\bysame, \emph{Le probl\`eme de la sch\'{e}matisation de {G}rothendieck
  revisit\'{e}}, \'{E}pijournal G\'{e}om. Alg\'{e}brique \textbf{4} (2020),
  Art. 14, 21. \MR{4191417}

\bibitem{wilkerson}
Clarence Wilkerson, \emph{Lambda-rings, binomial domains, and vector bundles
  over {$CP(\infty)$}}, Comm. Algebra \textbf{10} (1982), no.~3, 311--328.
  \MR{651605}

\bibitem{xantcha}
Qimh~Richey Xantcha, \emph{Binomial rings: axiomatisation, transfer and
  classification}, arXiv preprint arXiv:1104.1931 (2011).

\bibitem{yuan-integral}
Allen Yuan, \emph{Integral models for spaces via the higher {F}robenius}, J.
  Amer. Math. Soc. \textbf{36} (2023), no.~1, 107--175. \MR{4495840}

\end{thebibliography}

\medskip
\noindent
\textsc{Department of Mathematics, Northwestern University}\\
{\ttfamily antieau@northwestern.edu}

\end{document}